\documentclass[12pt]{article}
\usepackage{amssymb,amsmath,amsthm, amsfonts}

\usepackage{graphicx}
\usepackage{epsfig}
\usepackage{tikz}
\theoremstyle{plain}
\newtheorem{theorem}{Theorem}[section]
\newtheorem{lemma}{Lemma}[section]

\theoremstyle{definition}
\newtheorem{definition}{Definition}[section]

\numberwithin{equation}{section}

\allowdisplaybreaks \numberwithin{equation} {section}

\begin{document}
\title{Spatial inhomogeneity for a three-dimensional doubly degenerate nutrient system with indirect consumption  }
\author{Ai Huang$^{1}$, Xiangmao De-ji$^{1}$, Jing Li$^{2}$, \ Yifu Wang$^{1}$ \ \\[2mm]
 { \small $^{1}$School of Mathematics and Statistics, Beijing Institute of Technology,}\\
  { \small Beijing 100081, People's Republic of China}\\
 { \small $^{2}$College of Science, Minzu University of China,}\\
  { \small Beijing 100081, People's Republic of China}}
         \date{}
         \maketitle
\noindent{\bf Abstract}\ \ This paper investigates the global dynamics of a doubly degenerate nutrient-taxis system with indirect consumption:
\begin{equation*}
\left\{
\begin{aligned}
&u_{t}=\nabla \cdot (uv\nabla u)-\nabla \cdot (u^{2}v\nabla v)+\ell vw,&x\in \Omega,\, t>0,\\
& v_{t}=\Delta v-vw,&x\in \Omega,\, t>0,\\
&w_t=\Delta w-w+u,&x\in\Omega,t>0
\end{aligned}
\right.
\end{equation*}
posed on a smooth bounded domain $\Omega\subset\mathbb{R}^{3}$  with no-flux boundary conditions.
 It is shown that
   for suitably regular initial data $(u_0,v_0,w_0)$, the associated initial-boundary value problem admits a global weak solution.
   Furthermore, in an appropriate topological setting, this solution converges to an equilibrium $(u_\infty, 0,w_\infty)$ as $t\rightarrow \infty$. Notably, when $u_0$ is nonconstant and the mass of $
   v_0$ is sufficiently small, the limiting  profiles $u_{\infty}$ and $w_{\infty}$ are  are spatially nonhomogeneous,
      capturing emergent patterning in nutrient-depleted environments.
    A cornerstone of our analysis is the introduction of novel functional inequalities,
       which  provide estimates from below for the integral $\int_{\Omega}u^{k}v|\nabla u|^2$ with some $k>-1$.

  \noindent{\bf Key words: }\ Chemotaxis; \ doubly degenerate diffusion;  boundedness; \ indirect consumption;\ pattern formation

\noindent {\bf\em MSC}:~ 35K65, 35B36, 35Q92, 35K59, 92C17
\vskip   0.2cm \footnotetext[1]{E-mail addresses: wangyifu@bit.edu.cn(Y. Wang)} \setlength{\baselineskip}{20pt}

   \section{Introduction}
\setcounter{section}{1}\setcounter{equation}{0} \
Spatial patterning is a ubiquitous phenomenon in biological systems. Bacterial colonies exemplify this through their remarkable morphological adaptability to environmental constraints. Notably, studies of \it{Bacillus subtilis} \rm populations under nutrient-limited conditions reveal that such microbial systems often develop intricate spatial configurations, including snowflake like fractal aggregates (\cite{Fujikawa1992,Fujikawa1989,Fujikawa1990}). These patterns are dynamically regulated by key environmental factors such as substrate stiffness, nutrient gradients, and temperature. The systematic emergence of such complex architectures suggests the existence of universal organizational principles governing bacterial collective dynamics in resource-scarce environments.

Mathematical modeling and rigorous analysis play a central role in elucidating the underlying mechanisms driving colony pattern formation. In this context, Keller-Segel-type models are particularly  capable of capturing colony pattern formation as a self-organizing dynamical process (\cite{Herrero(1997),Hillen(2009)}). To  better understand the complex spatio-temporal dynamics
of the bacterium \it{Bacillus subtilis}  \rm  observed in experiments (\cite{Ohgiwari,Fujikawa1992,Fujikawa1989}), a
recent modeling approach   proposes modified Keller-Segel systems of the form
\begin{equation}\label{1.1a}
\left\{\begin{array}{ll}
		u_{t}=\nabla\cdot(uv\nabla u)-\chi \nabla\cdot(u^{\alpha}v\nabla v)+\ell uv,\\
		v_{t}=\Delta v-uv,\\
	\end{array}
		\right.
\end{equation}
with $\alpha=2, \chi\geq 0$ and $\ell\geq 0$ for the population density $u=u(x,t)$ and the food resource distribution $v=v(x,t)$ in \cite{Leyva2013,Plaza2019,Kawasaki1997}.
Here chemotactic motion of the bacterium \it{Bacillus subtilis} \rm can significantly enhance bacterial colony growth rates  observed in biological experiments and demonstrated through extensive numerical simulations \cite{Leyva2013,Plaza2019}.
 The key novelty of  system   \eqref{1.1a} consists in the appearance of nutrient density $v=v(x,t)$ as the factor not only in the cross-diffusion term, but especially also in the part related to random diffusion
of population density $u=u(x,t)$.
The reduction of bacterial motility near site of small nutrient concentrations herein seems in good accordance with experimentally
gained knowledge on bacterial migration in nutrient-poor environments.

From the perspective of mathematical analysis, the signal-dependent degeneracy in the first  equation  of
  \eqref{1.1a}
 evidently brings about significant challenges beyond those encountered in  well-understood Keller-Segel
 models which  coupled the chemotactic motion  to the density-dependent diffusion of standard porous medium-type
(\cite{Lankeit,Blanchet,WinklerMAS2025}).  It is also worth mentioning that, as the close relatives  of  \eqref{1.1a},  the taxis-type migration-consumption system of the form
\begin{equation}\label{1.2a}
\left\{\begin{array}{ll}
		u_{t}=\Delta(uv^\alpha),\\
		v_{t}=\Delta v-uv,\\
	\end{array}
		\right.
\end{equation}
with $\alpha\geq 1$ has been studied recently.  This system is relevant in the modeling of microbial migration processes
involving so-called local sensing of concentration levels of a directing chemical $v$ (\cite{Fu,Liu}).
Unlike \eqref{1.1a}, the feature of \eqref{1.2a} stems from  the precisely   quantifiable interplay  between random diffusive and cross-diffusive contributions,  encapsulated by a single Laplace operator. 
This facilitates a duality-based analytical strategy, enabling rigorous characterization of the nontrivial dynamics of initial-boundary value problems for \eqref{1.2a}, as partially evidenced 
in the literature \cite{Winkler(2024),Li Winkler(2022R),Winkler(2023Non),WinklerMAS2025,Laurenot}.


The double-degeneracy of the first equation  in  \eqref{1.1a} therein  apparently reduces a priori information on regularity to a significant extent, accordingly  already issues related to  
basic existence theories thereof seem far from evident, and
in particular
 it seems widely unclear yet how far the attractive taxis mechanism
in \eqref{1.1a}  may drive uncontrolled destabilization---potentially culminating in singularity formation.
It is  shown that in the taxis-free framework 
($\chi=0$), the associated initial-boundary value problem for \eqref{1.1a} possesses a global weak solution in the smooth convex domain $\Omega\subset \mathbb{R}^n $. Moreover, within an appropriate topological setting,
  the solution will approach the non-homogeneous  steady-state $(u_\infty,0)$ in the large time limit
(see \cite{WCVPDE}).
  It is noticed that due to the  absence of taxis effects in \eqref{1.1a}, the comparison principle
  allows for deriving local bounds for  $\|u(\cdot,t)\|_{L^\infty(\Omega)}$, and thereby
   the boundedness of $\int_0^\infty\int_\Omega u^{q-1}v|\nabla u|^2$ with $ q\in(0,1) $ becomes the basic regularity property of \eqref{1.1a}.  These estimates provide the essential foundation for establishing the global existence of  weak solution.

 To the best of our knowledge,  the available analytical results for \eqref{1.1a}
  remain restricted to low-dimensional settings so far:
   In one-dimensional setting, the dynamics of system \eqref{1.1a} with $\alpha=2$ have been  rather comprehensively characterized. 
 Global weak solutions were first constructed for arbitrarily large initial data under the integrability condition $\int_\Omega \ln u_0>-\infty$ \cite{Winkler3}, exhibiting asymptotic convergence to porous medium-type profiles. This constraint  were removed in  \cite{Li Winkler} by the analysis of the energy functional $\int_\Omega u\ln u+\int_\Omega \frac{v_x^2}{v}$. For the two-dimensional case, \cite{LiJDE} established the existence of global weak solutions in bounded convex domains with $1<\alpha<\frac3 2$. This result was later extended in  \cite{WJDE}  to all  $\alpha<2$, where  $L^{\infty}$-bounds for the solutions was also proved. Furthermore, the global boundedness of \eqref{1.1a} with $\alpha=2$ was established under a smallness condition on $v_0$  in \cite{WNARWA},  and this smallness condition was relaxed in \cite{ZhangLi}.

   In striking contrast to lower-dimensional cases, where effective embedding theorems facilitate the analysis,
    the findings on the three-dimensional version of \eqref{1.1a}  even at the level of global solvability theory seems  limited to 
   specific ranges of the exponent
    $\alpha$:   
The initial breakthrough was achieved in \cite{LiJDE}, which established a solvable regime of
 $\frac{7}{6} < \alpha < \frac{13}{9}$ for bounded convex domains.  Subsequent work \cite{Wu}
   suggested a potential extension to
   $1 < \alpha < \frac{3}{2}$.
Through functional inequality techniques, the solvable regime was further extended to $\frac{3}{2} < \alpha < \frac{19}{12}$ in \cite{XiangMao}, which also established non-trivial dynamical behaviors of solutions.
 In addition, the regularizing effect of
  logistic-type terms $\rho u-\mu u^k$ on
 system \eqref{1.1a} was investigated:
 Global existence of weak solutions was established for arbitrary dimensions $n\geq2$ when $k>\frac{n+2}{2}$ (\cite{XP}), and in particular  the continuity of weak solutions was achieved  for the case $n=k=2$ (\cite{LiWinkler2}).

In the classical Keller-Segel system, the chemical substance is  
 directly secreted by cells themselves. However, in some realistic situations
  the
chemotactic signal may be neither produced nor consumed by the cells themselves directly. In this regard, a well-studied example occurs in predator-prey ecosystems, where predators track chemical signals released by their prey (\cite{TelloWrzosek(2017),Strohm(2013),Surulescu(2021)}). 
   The typical form of the chemotaxis model involving indirect signal consumption is
\begin{equation}\label{3}
\begin{cases}
n_t=\Delta n-\nabla\cdot(n\nabla v),\\
v_t=\Delta v-vw,\\
w_t=\delta\Delta u-w+n.\\
\end{cases}
\end{equation}
At first glance, (\ref{3}) appears advantageous for establishing global boundedness of solutions, as its  second equation readily provides an $L^\infty$-bound for $v$ through standard comparison arguments. Nevertheless,
this bound alone is insufficient for dominating
 the destabilization of chemotactic  diffusion in the $n$-equation. Indeed, the existence of globally classical solutions is guaranteed only under small initial data $n_0$ even in three-dimensional settings (\cite{Fuest(2019)}). While substantial efforts have recently been devoted to analyzing  the impact of indirect signal mechanisms on the dynamics of the corresponding systems, most results are limited to either signal production scenarios (\cite{TaoWinkler(2017),Daicvpde(2023),Herrero(1997), Winkler(2013), Fujie(2017)}) or systems augmented with  the logistic source (\cite{Xing(2021),Liu Li Huang(2020)}).

 Motivated by the findings discussed above,  a natural question arises: to what extent does the indirect mechanism genuinely enhance the
 regularity of the three-dimensional  version of \eqref{1.1a}?
To this end, we consider the
following  doubly degenerate nutrient system with indirect consumption
\begin{equation}\label{1.1}
	\left\{\begin{array}{ll}
		u_{t}=\nabla\cdot(uv\nabla u)-\nabla\cdot(u^2v\nabla v)+\ell vw,& \quad x\in\Omega, t>0,\\
		v_{t}=\Delta v-vw,&\quad x\in\Omega, t>0,\\
		w_{t}=\Delta w-w+u,&\quad x\in\Omega, t>0,\\
	(uv\nabla u-u^2v\nabla v)\cdot\nu=\nabla v\cdot\nu=\nabla w\cdot\nu=0,
&\quad x\in\partial\Omega, t>0,\\
		u(\cdot,0)=u_{0},v(\cdot,0)=v_{0},w(\cdot,0)=w_{0},&\quad x\in\Omega,
	\end{array}
		\right.
\end{equation}
posed in a smoothly bounded domain $\Omega\subset\mathbb{R}^3$, and the parameter satisfies $\ell\geq0$. The initial data $(u_0,v_0,w_0)$ is assumed to satisfy
\begin{equation}\label{1.2}
    \left\{
    \begin{array}{l}
        u_0\in L^{\infty}(\Omega)~\mbox{is such that}~u_0\geq0~\mbox{and}~u_0\not\equiv0 ~\mbox{in}~\overline\Omega,\\
        v_0\in W^{1,\infty}(\Omega)~\mbox{is such that}~v_0>0~\mbox{in}~\overline\Omega,\\
        w_0\in W^{1,\infty}(\Omega)~\mbox{is such that}~w_0\geq0~\mbox{and}~w_0\not\equiv0~\mbox{in}~\overline\Omega
    \end{array}
    \right.
\end{equation}
as well as
\begin{equation}\label{1.6}
    \|u_0\|_{L^\infty(\Omega)}+\|v_0\|_{W^{1,\infty}(\Omega)}+\| w_0\|_{L^\infty(\Omega)}\leq K
\end{equation}
for some $K>0$.

In this context, the first result asserts  the global existence and  boundedness of continuous weak solutions to  \eqref{1.1}, without requiring convexity of domain or imposing any smallness conditions on the initial data.

\begin{theorem}\label{theorem 1.1}
	Let $\Omega\subset\mathbb{R}^{3}$ be a smoothly bounded domain and $\ell\geq0$. Then there exists  $C(K)>0$ with the property that whenever $u_0, v_0$ and $w_0$ fulfill \eqref{1.2} and \eqref{1.6}, the problem \eqref{1.1} admits a
 continuous global weak  solution $(u,v,w)$  in the sense of Definition \ref{Definition 2.1} below, which is such that
    \begin{equation}\label{1.7}
\left\{
\begin{aligned}
&u\in C^{0}_{loc}(\overline{\Omega}\times [0,\infty))\cap L^{\infty}_{loc}(\overline{\Omega}\times [0,\infty))\\
&v\in C^{0}_{loc}(\overline{\Omega}\times [0,\infty))\cap C^{2,1}_{loc}(\overline{\Omega}\times (0,\infty))\cap L^{\infty}_{loc}([0,\infty);W^{1,\infty}(\Omega))\\
&w\in C^{0}_{loc}(\overline{\Omega}\times [0,\infty))\cap C^{2,1}_{loc}(\overline{\Omega}\times (0,\infty))\\
\end{aligned}
\right.
\end{equation}
and that  $u\geq0, v>0$ and $w>0$ on $\overline{\Omega}\times(0,\infty)$, and that
\begin{equation}\label{1.8}
    \|u(\cdot,t)\|_{L^\infty(\Omega)}+\|v(\cdot,t)\|_{W^{1,\infty}(\Omega)}+\|w(\cdot,t)\|_{W^{1,\infty}(\Omega)}\leq C(K)~~\mbox{for all}~~t>0.
\end{equation}
\end{theorem}
Furthermore, based on the duality-based argument, the a priori estimates obtained in the proof of Theorem \ref{theorem 1.1} allow for the derivation of the following stability property in \eqref{1.1}.
\begin{theorem}\label{theorem1.2}
Let $\Omega \subset \mathbb R^3$ be a bounded domain with smooth boundary. Then for each $\eta>0$, there exists $\delta_1=\delta_1(\eta, K)>0$ whenever $u_0$, $v_0$ and $w_0$ fulfill
\eqref{1.2} and \eqref{1.6}, as well as
$$\int_{\Omega}v_0\leq \delta_1,$$
the solution $(u,v,w)$ of \eqref{1.1} obtained in Theorem \ref{theorem 1.1} satisfies
\begin{align}\label{1.9}
    \|u(\cdot,t)-u_0\|_{(W^{1,\infty}(\Omega))^*}\leq \eta~~\mbox{for all}~~t>0.
\end{align}
\end{theorem}

It is observed that  system \eqref{1.1} admits an uncountable family of steady states of the form
 $(u_*,0,w_*)$, where $u_*$ is a reasonably regular
 function and $w_*$ solves the elliptic equation  $-\Delta w+w=u_* $
 under no-flux boundary conditions.
  Beyond the stability  property  indicated in Theorem 1.2, as a remarkable feature of \eqref{1.1}, it is confirmed that every solution emanating from  the initial data of arbitrary size asymptotically approaches a steady state  $(u_\infty,0,w_\infty)$, and in particular the limit profiles  $u_\infty$ and $w_\infty$ are necessarily non-homogeneous  whenever $u_0\not\equiv const.$ and the initial concentration $v_0$   is sufficiently small.

\begin{theorem}\label{Th1.3}
Let $\Omega \subset \mathbb R^3$ be a bounded domain with smooth boundary, and $\ell\geq0$.
 Then  there exist  nonnegative functions $u_{\infty}\in C(\Omega)$
 and $w_{\infty}\in C^1(\Omega)$ such that as $t\to \infty$, the solution $(u,v,w)$ of \eqref{1.1} obtained in Theorem \ref{theorem 1.1} satisfies
\begin{align}\label{1.10}
    u(\cdot,t)\rightarrow u_{\infty}~~\mbox{in}~~L^\infty(\Omega),\quad v(\cdot,t)\rightarrow 0~~\mbox{in}~~W^{1,p}(\Omega)~~ \mbox{for all}~~p\geq 1,
\end{align}
and
\begin{align}
w(\cdot,t)\rightarrow w_{\infty}~~ \mbox{in}~~W^{1,\infty}(\Omega).
\end{align}
Moreover, for the given nonnegative function $u_0\not\equiv$ const., one can find
$\delta_2=\delta_2(K,u_0)>0$ such that whenever $u_0,$ $v_0$ and $w_0$ fulfill
\eqref{1.2} and \eqref{1.6}, as well as
$$\int_{\Omega}v_0\leq \delta_2,$$
the limit function satisfies
 $ u_{\infty}\not\equiv$ const..
\end{theorem}


The most crucial  step in our approach toward the basic solution theory for   \eqref{1.1} is to establish the global boundedness of
$\int_\Omega u^{\frac3 2}$ through a self-mapping type argument (see Lemma \ref{lemma3.5}).  In this direction, the challenge herein is  how to make appropriate use of the degenerate  action   of  the form $\int_{\Omega}u^{k}v|\nabla u|^2$. Indeed,
the cornerstone of this argument
 lies in effectively controlling
  $\int_{\Omega} u^{\beta} v$  ($\beta>\frac 52$) 
by the dissipated quantities   $\int_{\Omega}u^{k}v|\nabla u|^2$ (for suitable $k$)
and  the signal-weighted gradient term $\int_\Omega\frac{|\nabla v|^{\frac{2\beta}{\beta-k-2}}}{v^{\frac{2\beta}{\beta-k-2}-1}}$.
This  will be achieved   by utilizing  
a functional inequality of the form
$$\int_{\Omega}\varphi^{\beta}\psi\leq  2\int_{\Omega}\varphi^{k}\psi|\nabla \varphi|^2+C_1(k, \beta)\int_\Omega\frac{|\nabla \psi|^{\frac{2\beta}{\beta-k-2}}}{\psi^{\frac{2\beta}{\beta-k-2}-1}}+C_2(k, \beta, \int_\Omega \varphi)\int_{\Omega}\varphi \psi$$
 for smooth $\varphi> 0$ and $\psi>0$ and any $k>-1, \beta\in (k+2,k+\frac8 3)$ with some $C_1(k, \beta)>0, C_2(k, \beta, \int_\Omega \varphi)>0$ (Lemma \ref{lemma3.4}).  Accordingly, one can derive the uniform-in-time boundedness of $\int_\Omega u^{\frac3 2}$ by a loop-type argument (Lemma \ref{lemma3.5}), thanks to  the exponential decay of $v$
 (Lemma \ref{lemma2.3}) and the weighted estimate for $\int_\Omega |\nabla v|^q v^{-q+1}$ (Lemma \ref{lemma2.7}). Thereafter, on the basis  of improved integrability properties of $u$, we establish its $L^p$ boundedness (Lemma \ref{lemma3.7})  with the help of
the conditional functional  inequality:
  $$\int_{\Omega}\varphi^{\beta} \psi\leq \eta \int_{\Omega}\varphi^{p-1}\psi|\nabla
    \varphi|^2+\eta \int_{\Omega}\varphi\frac{|\nabla  \psi|^q}{{\psi}^{q-1}}+C(L)\int_{\Omega}\varphi \psi$$
  is valid with some $C(L)>0$ for any smooth functions $\varphi\geq  0$ and $\psi>0$ fulfilling
   $\int_{\Omega}\varphi^{\frac{3}{2}}\leq L$, $\eta>0$,
$p>1$, $q>4$ and $\beta\in [1,p+2)$ (Lemma \ref{lemma3.6}).
  To characterize the large-time behavior of solutions $(u,v,w)$ to \eqref{1.1}, we develop a temporal analysis framework that carefully accounts for time-dependent dynamics. The cornerstone of our approach involves duality-based
  estimation techniques, through which we achieve the large  time stabilization feature:
 \begin{equation}\label{1.12}
 \int_0^{\infty}\|u_t(\cdot,t)\|_{(W^{1,\infty}(\Omega))^*}dt\leq C(K)\cdot\left\{\int_{\Omega} v_0\right\}^\sigma,
 \end{equation}
where $C(K)>0$ and $\sigma>0$ are constants (Lemma \ref{lemma4.1}). The inequality \eqref{1.12} quantitatively captures how the initial nutrient distribution $v_0$ controls the long-term regularity of the solution components. In consequence, despite
  \eqref{1.1} has  the feature of double degeneracy and particularly possesses the  uncountable set of steady states,  \eqref{1.1}  generates large time behaviour less chaotic  than  that in  some cases studied in \cite{PS2002,PY2003,WJDDE}. Indeed,  our result asserts that
  for any suitably regular initial data, the corresponding
 solution of \eqref{1.1}  stabilizes toward  a steady state $(u_{\infty},0,w_{\infty})$ as $t\to \infty$.

  The structure of this paper is as follows: In Section 2,  we specify the weak solutions of system \eqref{1.1} and
 establishes fundamental a priori estimates for the regularized systems. Building upon these estimates and further higher regularity properties of the regularized systems, we then
    prove the global existence of $(u,v,w)$ via compactness arguments.
Section 3 is devoted to the derivation of the $L^{\frac3 2}$-boundedness for the first component $u_\varepsilon$ of
the approximated solutions through a  loop-type argument. 
 In Section 4, we investigates the stability properties of $(u,v,w)$ by means of a duality-based argument.
\section{Preliminaries}
In view of the fact that the diffusion mechanism in \eqref{1.1} contains a degeneracy of porous medium type, our existence theory will be carried out in the framework of the natural generalized solution concept specified as follows.
\begin{definition}\label{Definition 2.1} Assume  that \eqref{1.2} holds. Then a pair of nonnegative functions
\begin{equation*}
\left\{
\begin{aligned}
&u\in C^{0}(\overline{\Omega}\times [0,\infty)), \\
&v\in C^{0}(\overline{\Omega}\times [0,\infty))\cap C^{2,1}(\overline{\Omega}\times(0,\infty)),\\
&w\in C^{0}(\overline{\Omega}\times [0,\infty))\cap C^{2,1}(\overline{\Omega}\times(0,\infty))\\
\end{aligned}
\right.
\end{equation*}
such that
\begin{align}\label{2.1}
  u^2 \in L^{1}_{loc}([0,\infty); W^{1,1}(\Omega))\quad \hbox{and}\quad u^2\nabla v \in L^{1}_{loc}(\overline{\Omega}\times
  [0,\infty);\mathbb{R}^3)
\end{align}
 will be called a continuous weak solution of \eqref{1.1} if
\begin{align}\label{2.2}
    -\int_{0}^{\infty}\!\int_{\Omega}u\varphi_t-\int_{\Omega}u_0\varphi(\cdot,0)=
    &-\frac{1}{2}\int_{0}^{\infty} \!\int_{\Omega}
     v \nabla u^2 \cdot\nabla\varphi 
     +\int_{0}^{\infty}\!\int_{\Omega}u^{2}v\nabla v\cdot \nabla \varphi
     +\ell\int_{0}^{\infty}\!\int_{\Omega}vw\varphi
\end{align}
for all
 $\varphi \in C_0^{\infty}(\overline{\Omega}\times [0,\infty))$
 fulfilling $\frac{\partial \varphi}{\partial \nu}=0$ on $\partial\Omega\times(0,\infty)$, and if
\begin{align}\label{2.3}
\int_{0}^{\infty}\int_{\Omega}v\varphi_t+\int_{\Omega}v_0\varphi(\cdot,0)=\int_{0}^{\infty}\int_{\Omega}\nabla
v\cdot\nabla\varphi+\int_{0}^{\infty}\int_{\Omega}vw\varphi
\end{align}
for any $\varphi \in C_0^{\infty}(\overline{\Omega}\times [0,\infty))$, as well as
\begin{equation}\label{2.4}
    \int_{0}^{\infty}\int_{\Omega}w\varphi_t+\int_{\Omega}w_0\varphi(\cdot,0)=\int_{0}^{\infty}\int_{\Omega}\nabla
w\cdot\nabla\varphi-\int_{0}^{\infty}\int_{\Omega}w\varphi+\int_{0}^{\infty}\int_{\Omega}u\varphi
\end{equation}
for all $\varphi\in C_0^{\infty}(\overline{\Omega}\times[0,\infty))$.
\end{definition}

Next, in order to construct such  weak solutions to \eqref{1.1}, we adopt an approximate approach only through a modification in the initial data. 
Specifically, for $\varepsilon\in(0,1)$, we consider the regularized variant of \eqref{1.1} given by
\begin{equation}\label{2.5}
    \left\{\begin{array}{ll}
         u_{\varepsilon t}=\nabla\cdot(u_\varepsilon v_\varepsilon \nabla u_\varepsilon )-\nabla\cdot(u_\varepsilon ^2v_\varepsilon \nabla v_\varepsilon )+\ell v_\varepsilon w_\varepsilon ,& \quad x\in\Omega, t>0,\\
		v_{\varepsilon t}=\Delta v_\varepsilon -v_\varepsilon w_\varepsilon ,&\quad x\in\Omega, t>0,\\
		w_{\varepsilon t}=\Delta w_\varepsilon -w_\varepsilon +u_\varepsilon ,&\quad x\in\Omega, t>0,\\
	\frac{\partial u_\varepsilon}{\partial\nu}=\frac{\partial v_\varepsilon}{\partial\nu}=\frac{\partial w_\varepsilon}{\partial\nu}=0,
&\quad x\in\partial\Omega, t>0,\\
		u_\varepsilon (\cdot,0)=u_{0}(x)+\varepsilon ,v_\varepsilon (\cdot,0)=v_{0}(x),w_\varepsilon (\cdot,0)=w_{0}(x),&\quad x\in\Omega,
    \end{array}
    \right.
\end{equation}
which, according to Lemma 2.1 of \cite{WNARWA}, admits globally defined classical solutions enjoying  a handy extensibility  criterion: 

\begin{lemma}\label{lemma2.1}
Let $\Omega \subset \mathbb R^3$ be a bounded domain with smooth boundary, $\ell\geq 0$, and assume that \eqref{1.2} holds.  Then for each $\varepsilon\in(0,1)$, there exists
$T_{max,\varepsilon}\in(0,\infty]$ and functions
\begin{equation*}
\left\{
\begin{aligned}
&u_{\varepsilon }\in C^0(\overline{\Omega}\times [0,T_{max,\varepsilon}))\cap C^{2,1}(\overline{\Omega}\times(0,T_{max,\varepsilon})), \\
&v_{\varepsilon }\in  \textstyle\bigcap_{q>3} C^0([0,T_{max,\varepsilon});W^{1,q}(\Omega))\cap C^{2,1}(\overline{\Omega}\times(0,T_{max,\varepsilon})),\\
&w_{\varepsilon }\in  C^0(\overline{\Omega}\times[0,T_{max,\varepsilon}))\cap C^{2,1}(\overline{\Omega}\times(0,T_{max,\varepsilon})) \\
\end{aligned}
\right.
\end{equation*}
such that $u_ \varepsilon >0$, $v_ \varepsilon >0$ and $w_{\varepsilon}>0$ in $\overline{\Omega} \times [0,T_{max,\varepsilon})$, that $(u_{\varepsilon
},v_{\varepsilon },w_{\varepsilon })$ solves \eqref{2.5} classically in $\Omega\times(0,T_{max,\varepsilon})$, and that
  \begin{equation} \label{2.6}
   \hbox{if}~~  T_{max,\varepsilon}<\infty, ~~ \mbox{then}~~
   \limsup_{t\nearrow T_{max,\varepsilon}}\|u_\varepsilon(\cdot , t)\|_{L^{\infty}(\Omega)}=\infty.
\end{equation}
Furthermore, this solution satisfies
\begin{align}\label{2.7}
    \|v_{\varepsilon}(\cdot , t)\|_{L^{\infty}(\Omega)}\leq \|v_{0 }\|_{L^{\infty}(\Omega)}\qquad \mbox{for all}~~ t\in(0,T_{max,\varepsilon}),
\end{align}
\begin{align}\label{2.8}
    \int_{\Omega}u_{0}\leq \int_{\Omega}u_{\varepsilon}(\cdot,t)\leq m:=\int_\Omega u_0+\ell\int_\Omega v_0\qquad \mbox{for
    all}~~ t\in(0,T_{max,\varepsilon})
\end{align}
and
\begin{align}\label{2.9}
    \int_{\Omega}w_{\varepsilon}(\cdot,t)\leq\max\bigg\{m,\int_\Omega w_0\bigg\}\qquad \mbox{for
    all}~~ t\in(0,T_{max,\varepsilon})
\end{align}
as well as
\begin{align}\label{2.10}
    \int_{0}^{T_{max,\varepsilon}}\int_{\Omega}w_{\varepsilon }v_{\varepsilon }\leq \int_{\Omega}v_{0}.
\end{align}
\end{lemma}
\begin{proof}
By the arguments in Lemma 2.1 in \cite{WNARWA} with possible minor modifications, we can derive \eqref{2.6}, while employing the maximum principle as well as straightforward integration in the equations of \eqref{2.5} we readily obtain \eqref{2.7}--\eqref{2.10}.
\end{proof}

By the known result of the Neumann heat semigroup $(e^{t\Delta})_{t\geq0}$ in $\Omega$,
it seems favorable to achieve lower bounds for $w_\varepsilon$.


\begin{lemma}\label{lemma2.2}Let
$t_0:=\min\{1,\frac{T_{\max,\varepsilon}}{6}\}$.
There exists constant $\kappa=\kappa(t_0)>0$ independent of $\varepsilon$ such that
$$w_\varepsilon(x,t)\geq\kappa~~\mbox{for all}~~x\in\Omega~\mbox{and}~t\in(2t_0,T_{\max,\varepsilon}).$$
\end{lemma}
\begin{proof}
According to the known result of Neumann heat semigroup $(e^{t\Delta})_{t\geq0}$ in $\Omega$ (\cite{Hillen(2013)}, Lemma 3.1), for given $\tau>0$, there exists $C(\tau)>0$ such that
\begin{equation}\label{2.11*}
e^{t\Delta}\phi\geq C(\tau)\int_{\Omega}\phi~~\mbox{for all}~~t\geq\tau~~\mbox{and each non-negative}~\phi\in C^0(\overline{\Omega}).
\end{equation}
By the variation-of-constant formula to the third equation of \eqref{2.5}, we have
\begin{equation}\label{2.11}
w_\varepsilon(\cdot,t)=e^{t(\Delta-1)}w_0(\cdot)+\int_0^t e^{(t-s)(\Delta-1)}u_\varepsilon(\cdot,s)ds~~\mbox{for all}~~t\geq0.
\end{equation}
Therefore, using  \eqref{2.8} and \eqref{2.11*}, we obtain that for $t\in(2t_0,T_{\max,\varepsilon})$
\begin{equation*}
\begin{split}
w_\varepsilon(\cdot,t)&=e^{t(\Delta-1)}w_0(\cdot)+\int_0^t e^{(t-s)(\Delta-1)}u_\varepsilon(\cdot,s)ds\\
&\geq\int_{0}^{t-t_0} e^{(t-s)(\Delta-1)}u_\varepsilon(\cdot,s)ds\\
&\geq C(t_0)\int_{0}^{t-t_0} e^{-(t-s)}\int_{\Omega}u_\varepsilon(\cdot,s)dxds\\
&\geq C(t_0)\|u_0\|_{L^1(\Omega)}\cdot\int_{0}^{t-t_0} e^{-(t-s)}ds\\
&\geq C(t_0)\|u_0\|_{L^1(\Omega)}(e^{-t_0}-e^{-2t_0})=:\kappa.
\end{split}
\end{equation*}
This completes proof.
\end{proof}
With the above statement at hand, we can derive the exponential decay property for the solution component $v_\varepsilon$, which plays a crucial role in our subsequent analysis.
\begin{lemma}\label{lemma2.3}
Let
$t_0:=\min\{1,\frac{T_{\max,\varepsilon}}{6}\}$ and $\kappa:=\kappa(t_0)>0$ as defined in Lemma \ref{lemma2.2}. Then for all $t\in(0,T_{\max,\varepsilon})$, we have
\begin{equation}\label{2.12}
v_\varepsilon(x,t)\leq e^{2\kappa t_0}\|v_0\|_{L^{\infty}(\Omega)}e^{-\kappa t}.
\end{equation}
\end{lemma}
\begin{proof}
Due to $w_\varepsilon(x,t)\geq\kappa$ for all $x\in\Omega$ and $t\in(2t_0,T_{\max,\varepsilon})$ by Lemma \ref{lemma2.2}, we have
$$v_{\varepsilon t}=\Delta v_\varepsilon-w_\varepsilon v_\varepsilon\leq\Delta v_\varepsilon-\kappa v_\varepsilon~~\mbox{for all}~~t\in(2t_0,T_{\max,\varepsilon}).$$
By means of a straightforward comparison with the ODE
$$\overline{v_{\varepsilon}}_t=-\kappa \overline{v_\varepsilon},~~~~\overline{v_\varepsilon}(2t_0)=\|v_0\|_{L^\infty(\Omega)},$$
we get
$$v_\varepsilon(x,t)\leq \overline{v_\varepsilon}(t)=
e^{2\kappa t_0}\|v_0\|_{L^\infty(\Omega)}e^{-\kappa t}~~\mbox{for all}~~t\in(2t_0,T_{\max,\varepsilon}).$$
 On the other hand, it follows from  \eqref{2.7} that for $t\in(0,2t_0)$,
\begin{equation*}
v_\varepsilon(\cdot,t)\leq\|v_0\|_{L^{\infty}(\Omega)}\leq e^{2\kappa t_0}\|v_0\|_{L^{\infty}(\Omega)}e^{-\kappa t}
\end{equation*}
due to $e^{\kappa(2t_0-t)}\geq1$.
Hence, we arrive at \eqref{2.12} readily.
\end{proof}
It is observed  that according to  $L^p-L^q$ estimates for the Neumann heat semigroup (\cite{Winkler(2010)}), $ \|w_\varepsilon(\cdot,t)\|_{L^{\mu}(\Omega)}$ with suitable $\mu>0$ can be controlled by $\|u_\varepsilon(\cdot,t)\|_{L^p(\Omega)}$.
\begin{lemma}\label{lemma2.4}
For $p\geq 1$ and  $\mu<\frac{3p}{(3-2p)_+}$, 
there exists constant $C(p,\mu)>0$ such that for all $t\in(0,T_{\max,\varepsilon})$,
\begin{equation}\label{2.14}
    \|w_\varepsilon(\cdot,t)\|_{L^{\mu}(\Omega)}\leq C(p,\mu)(\|w_{0}\|_{L^{\mu}(\Omega)}+\sup\limits_{0<s<T_{\max,\varepsilon}}\|u_\varepsilon(\cdot,s)\|_{L^p(\Omega)}).
\end{equation}

\end{lemma}
\begin{proof}[Proof]
	By means of the variation-of-constants formula for $w_\varepsilon$, we have
	\begin{equation*}
		w_\varepsilon(\cdot,t)=e^{t(\Delta-1)}w_{0}+\int_{0}^{t}e^{(t-s)(\Delta-1)}u_\varepsilon(\cdot,s)ds.
	\end{equation*}
Letting $\overline{u_\varepsilon}(t) = \frac{1}{|\Omega|}\int_\Omega u_\varepsilon(x,t) dx$. Then by the Neumann heat semigroup estimates in \cite{Winkler(2010)} and  \eqref{2.8}, one can find constants $c_i>0, i=1,2,3$,
 such that
\begin{align*}
		&\|w_\varepsilon(\cdot,t)\|_{L^{\mu}(\Omega)}\\
\leq&\|e^{t(\Delta-1)}w_{0}\|_{L^{\mu}(\Omega)}+\int_{0}^{t}\|e^{(t-s)(\Delta-1)}(u_\varepsilon-\overline{u_\varepsilon})(\cdot,s)\|_{L^{\mu}(\Omega)}ds+\int_{0}^{t}\|e^{(t-s)(\Delta-1)}\overline{u_\varepsilon}(s)\|_{L^{\mu}(\Omega)}ds\\
\leq & c_1\|w_{0}\|_{L^{\mu}(\Omega)}+
c_1\int_{0}^{t}(1+(t-s)^{-\frac{3}{2}(\frac1 p-\frac{1}{\mu})})e^{-(t-s)}\|u_\varepsilon-\overline{u_\varepsilon}\|_{L^{p}(\Omega)}ds+
\frac{c_2}{|\Omega|}\int_\Omega u_\varepsilon\int_0^t e^{-(t-s)}ds\\
		\leq & c_1\big\|w_{0}\|_{L^{\mu}(\Omega)}+\sup\limits_{0<s<T_{\max,\varepsilon}}\|u_\varepsilon(\cdot,s)\|_{L^p(\Omega)}\cdot\bigg(c_3\int_{0}^{\infty}(1+\sigma^{-\frac{3}{2}(\frac1 p-\frac{1}{\mu})})e^{-\sigma}d\sigma+c_2|\Omega|^{-\frac1 p}\bigg)\\
		\leq &
C(p,\mu)(\|w_{0}\|_{L^{\mu}(\Omega)}+\sup\limits_{0<s<T_{\max,\varepsilon}}\|u_\varepsilon(\cdot,s)\|_{L^p(\Omega)})
\end{align*}
with $C(p,\mu):=\{c_1,c_2|\Omega|^{-\frac 1p}+c_3\int_{0}^{\infty}(1+\sigma^{-\frac{3}{2}(\frac1 p-\frac{1}{\mu})})e^{-\sigma}d\sigma\}<\infty$. 
\end{proof}
As a crucial step in our subsequent analysis, we establish a local point-wise lower bound estimate for $v_\varepsilon$, although $v_\varepsilon$ decays exponentially as $t\to\infty$ asserted in  Lemma \ref{lemma2.3}.
\begin{lemma}\label{lemma2.5}
For all $T>0$, there exists $C(T)>0$ such that
\begin{equation}\label{2.22}
v_\varepsilon(\cdot,t)\geq C(T)~~\mbox{for all}~~t\in(0,T)\cap(0,T_{\max,\varepsilon}).
\end{equation}
\end{lemma}
\begin{proof}[Proof]
An application of Lemma \ref{lemma2.4} to $p:=1$ and $\mu:=2$, yields $c_1>0$ such that
 \begin{equation}\label{2.16*}
 \|w_\varepsilon(\cdot,t)\|_{L^2(\Omega)}\leq c_1(1+m).
 \end{equation}
Let
$z_\varepsilon(x,t):=-\ln\frac{v_\varepsilon(x,t)}{\|v_0\|_{L^\infty(\Omega)}}.$
Then the second equation of \eqref{2.5}  becomes
\begin{equation}\label{2.23}
         z_{\varepsilon t}=\Delta z_\varepsilon-|\nabla z_\varepsilon|^2+w_\varepsilon,\\
    \end{equation}
      with the initial value $z_0(x)=-\ln\frac{v_0(x)}{\|v_0\|_{L^\infty(\Omega)}}$.
	According to the variation-of-constants formula to \eqref{2.23}, we arrive at
	\begin{equation*}
		\begin{split}
			z_\varepsilon(\cdot,t)&=e^{t\Delta}z_0(\cdot)-\int_{0}^{t}e^{(t-s)\Delta}|\nabla z_\varepsilon(\cdot,s)|^{2}ds+\int_{0}^{t}e^{(t-s)\Delta}w_\varepsilon(\cdot,s)ds\\
&\leq e^{t\Delta}z_0(\cdot)+\int_{0}^{t}e^{(t-s)\Delta}w_\varepsilon(\cdot,s)ds.
		\end{split}
	\end{equation*}
 Hence from the nonnegativity of $z_\varepsilon(x,t)$, \eqref{2.16*} and \eqref{2.9}, it follows that for all $t\in(0,T)\cap(0,T_{\max,\varepsilon})$,
	\begin{align}\label{2.24}
		&\|z_\varepsilon(\cdot,t)\|_{L^{\infty}(\Omega)}\nonumber\\
\leq
&\|e^{t\Delta}z_{0}\|_{L^{\infty}(\Omega)}+\int_{0}^{t}\|e^{(t-s)\Delta}
w_\varepsilon(\cdot,s)\|_{L^{\infty}(\Omega)}ds\nonumber\\
			\leq
&\|z_{0}\|_{L^{\infty}(\Omega)}+\int_{0}^{t}\|e^{(t-s)\Delta}(w_\varepsilon-\overline{w_\varepsilon})(\cdot,s)
\|_{L^{\infty}(\Omega)}ds+\int_{0}^{t}\|e^{(t-s)\Delta}
\overline{w_\varepsilon}(s)\|_{L^{\infty}(\Omega)}ds\\
			\leq
&
\|z_{0}\|_{L^{\infty}(\Omega)}+c_2\int_{0}^{t}(1+(t-s)^{-\frac{3}{4}})e^{-\lambda_{1}(t-s)}\|(w_\varepsilon-\overline{w_\varepsilon})(\cdot,s)\|_{L^{2}(\Omega)}ds+ \int_{0}^{t}\overline{w_\varepsilon}(s)ds\nonumber\\
\leq&\|z_{0}\|_{L^{\infty}(\Omega)}+c_3\int_{0}^{\infty}(1+\sigma^{-\frac{3}{4}})e^{-\lambda_{1}\sigma}d\sigma+\frac 1{|\Omega|}\max\bigg\{m,\int_{\Omega}w_0\bigg\}t,\nonumber
	\end{align}
where $\overline{w_\varepsilon}(t):=\frac{1}{|\Omega|}\int_{\Omega}w_\varepsilon(x,t)dx$, for constants $c_2>0$ and $c_3>0$.
Therefore we  get
\begin{equation*}
			z_\varepsilon(x,t)\leq \varsigma(1+t)
\end{equation*}
with $\varsigma:=\max\{\|z_{0}\|_{L^{\infty}(\Omega)} + c_3 \int_{0}^{\infty}(1+\sigma^{-\frac{3}{4}})e^{-\lambda_{1}\sigma}d\sigma,
 \frac 1{|\Omega|}\max\{m,\int_{\Omega}w_0\}\}$. In conjunction with the definition of $z_\varepsilon$, we have
 \begin{equation}\label{2.19*}
 v_\varepsilon(x,t)\geq \|v_0\|_{L^\infty(\Omega)}e^{-\varsigma(1+t)},
 \end{equation}
and thus arrive at \eqref{2.22} immediately.
	\end{proof}

Similar as in Lemma \ref{lemma2.4},  the estimate of $\|v_\varepsilon(\cdot,t)\|_{W^{1,\theta}(\Omega)}$ can be established by employing the Neumann heat semigroup estimates once more, and accordingly the proof thereof is omitted herein.
\begin{lemma}\label{lemma2.6}
Suppose  $(u_\varepsilon,v_\varepsilon,w_\varepsilon)$  be the solution of \eqref{2.5}. Then for 
\begin{equation*}
    \begin{aligned}
        \theta\in\left\{
    \begin{array}{l}
        {\left({1,\frac{3p}{(3-p)_+}}\right)},\quad p\leq 3, \\
         {(1,\infty]},\quad\quad p>3,
    \end{array}\right.
\end{aligned}
\end{equation*}
there exists constant $C(p,\theta)>0$ such that  for all $t\in(0,T_{\max,\varepsilon})$
\begin{equation*}
    \|v_\varepsilon(\cdot,t)\|_{W^{1,\theta}(\Omega)}\leq C(p,\theta)(1+\sup\limits_{0<t<T_{\max,\varepsilon}}\|w_\varepsilon(\cdot,t)\|_{L^p(\Omega)}).
\end{equation*}\end{lemma}
The core of our analysis consists of appropriately controlling the taxis-driven contributions through the diffusion-induced dissipation represented by the integral
$\int_{\Omega} u_\varepsilon^{k} v_\varepsilon |\nabla u_\varepsilon|^2$. Notably, due to the decay of the weight $v_\varepsilon$,
this
 requires us to explore the evolution of the functional of the form  $\int_\Omega v_\varepsilon^{-q+1}|\nabla v_\varepsilon|^q$,
and
  adequately take advantage of corresponding singularly weighted dissipation rates arising therein.
\begin{lemma}\label{lemma2.7}
    Let $\Omega\subset\mathbb{R}^3$ be a smoothly bounded domain and $q\geq2$. Then for all $t\in(0,T_{\max,\varepsilon})$, there exist $\gamma(q)>0$ and $C(q)>0$ such that
    \begin{equation}\label{2.15}
    \begin{split}
        &\frac{d}{dt}\int_{\Omega}v_\varepsilon^{-q+1}|\nabla v_\varepsilon|^q+\gamma(q)\int_{\Omega}v_\varepsilon^{-q-1}|\nabla v_\varepsilon|^{q+2}\leq  C(q)\bigg(\int_{\Omega}w_\varepsilon^{\frac{q+2}{2}}v_\varepsilon+\int_\Omega v_\varepsilon\bigg)
    \end{split}
    \end{equation}
    where $\gamma(q):=\frac{q}{4(q+\sqrt{3})^2}$.
\end{lemma}
\begin{proof}
    According to the second equation in \eqref{2.5} as well as the identities $\nabla v_\varepsilon\cdot\nabla\Delta v_\varepsilon=\frac{1}{2}\Delta|\nabla v_\varepsilon|^2-|D^2v_\varepsilon|^2$ and $\nabla|\nabla v_\varepsilon|^2=2D^2v_\varepsilon\cdot\nabla v_\varepsilon$, several integrations by the parts show that
    \begin{align}\label{2.16}
           & \frac{d}{dt}\int_\Omega v_\varepsilon^{-q+1}|\nabla v_\varepsilon|^q\nonumber\\
            &=q\int_\Omega v_\varepsilon^{-q+1}|\nabla v_\varepsilon|^{q-2}\nabla v_\varepsilon\cdot\nabla\{\Delta v_\varepsilon-v_\varepsilon w_\varepsilon\}-(q-1)\int_\Omega v_\varepsilon^{-q}|\nabla v_\varepsilon|^q\cdot\{\Delta v_\varepsilon-v_\varepsilon w_\varepsilon\}\nonumber\\
            &=-\frac{q(q-2)}{4}\int_\Omega v_\varepsilon^{-q+1}|\nabla v_\varepsilon|^{q-4}|\nabla|\nabla v_\varepsilon|^2|^2+q(q-1)\int_\Omega v_\varepsilon^{-q}|\nabla v_\varepsilon|^{q-2}\nabla v_\varepsilon\cdot\nabla|\nabla v_\varepsilon|^2\nonumber\\
            &-q\int_\Omega v_\varepsilon^{-q+1}|\nabla v_\varepsilon|^{q-2}|D^2 v_\varepsilon|^2-q(q-1)\int_\Omega v_\varepsilon^{-q-1}|\nabla v_\varepsilon|^{q+2}\\
            &+\frac{q}{2}\int_{\partial\Omega}v_\varepsilon^{-q+1}|\nabla v_\varepsilon|^{q-2}\cdot\frac{\partial|\nabla v_\varepsilon|^2}{\partial\nu}+q(q-2)\int_\Omega w_\varepsilon v_\varepsilon^{-q+2}|\nabla v_\varepsilon|^{q-4}\nabla v_\varepsilon\cdot(D^2 v_\varepsilon\cdot\nabla v_\varepsilon)\nonumber\\
            &+q\int_\Omega w_\varepsilon v_\varepsilon^{-q+2}|\nabla v_\varepsilon|^{q-2}\Delta v_\varepsilon-(q-1)^2\int_\Omega w_\varepsilon v_\varepsilon^{-q+1}|\nabla v_\varepsilon|^q\nonumber
    \end{align}
    for all $t\in(0,T_{\max,\varepsilon})$. Due to the fact that
    \begin{equation*}
        |D^2\ln\varphi|^2=\frac{1}{\varphi^2}|D^2\varphi|^2-\frac{2}{\varphi^3}\nabla\varphi\cdot(D^2\varphi\cdot\nabla\varphi)+\frac{1}{\varphi^4}|\nabla\varphi|^4~~\mbox{for all}~~\varphi\in C^2(\Omega),
    \end{equation*}
we derive that
    \begin{equation*}
        |D^2v_\varepsilon|^2=v_\varepsilon^2|D^2\ln v_\varepsilon|^2+\frac{1}{v_\varepsilon}\nabla v_\varepsilon\cdot\nabla|\nabla v_\varepsilon|^2-\frac{1}{v_\varepsilon^2}|\nabla v_\varepsilon|^4,
    \end{equation*}
     and then we can rewrite
    \begin{align}\label{2.17}
            &-\frac{q(q-2)}{4}\int_\Omega v_\varepsilon^{-q+1}|\nabla v_\varepsilon|^{q-4}|\nabla|\nabla v_\varepsilon|^2|^2+q(q-1)\int_\Omega v_\varepsilon^{-q}|\nabla v_\varepsilon|^{q-2}\nabla v_\varepsilon\cdot\nabla|\nabla v_\varepsilon|^2\nonumber\\
            &-q\int_\Omega v_\varepsilon^{-q+1}|\nabla v_\varepsilon|^{q-2}|D^2 v_\varepsilon|^2-q(q-1)\int_\Omega v_\varepsilon^{-q-1}|\nabla v_\varepsilon|^{q+2}\nonumber\\
            &=-q\int_\Omega v_\varepsilon^{-q+3}|\nabla v_\varepsilon|^{q-2}|D^2\ln v_\varepsilon|^2-\frac{q(q-2)}{4}\int_\Omega v_\varepsilon^{-q+1}|\nabla v_\varepsilon|^{q-4}|\nabla|\nabla v_\varepsilon|^2|^2\\
            &+q(q-2)\int_\Omega v_\varepsilon^{-q}|\nabla v_\varepsilon|^{q-2}\nabla v_\varepsilon\cdot\nabla|\nabla v_\varepsilon|^2-q(q-2)\int_\Omega v_\varepsilon^{-q-1}|\nabla v_\varepsilon|^{q+2}\nonumber\\
            &=-q\int_\Omega v_\varepsilon^{-q+3}|\nabla v_\varepsilon|^{q-2}|D^2\ln v_\varepsilon|^2-\frac{q(q-2)}{4}\int_\Omega v_\varepsilon^{-q+1}|\nabla v_\varepsilon|^{q-4}\bigg|\nabla|\nabla v_\varepsilon|^2-\frac{2}{v_\varepsilon}|\nabla v_\varepsilon|^2\nabla v_\varepsilon\bigg|^2,\nonumber
    \end{align}
    while estimating $|\Delta v_\varepsilon|\leq\sqrt{3}|D^2v_\varepsilon|$, we obtain that
    \begin{equation}\label{2.18}
    \begin{split}
        &q(q-2)\int_\Omega w_\varepsilon v_\varepsilon^{-q+2}|\nabla v_\varepsilon|^{q-4}\nabla v_\varepsilon\cdot(D^2 v_\varepsilon\cdot\nabla v_\varepsilon)+q\int_\Omega w_\varepsilon v_\varepsilon^{-q+2}|\nabla v_\varepsilon|^{q-2}\Delta v_\varepsilon\\
        &\leq q(q-2)\int_\Omega w_\varepsilon v_\varepsilon^{-q+2}|\nabla v_\varepsilon|^{q-2}|D^2v_\varepsilon|+q\sqrt{3}\int_\Omega w_\varepsilon v_\varepsilon^{-q+2}|\nabla v_\varepsilon|^{q-2}|D^2v_\varepsilon|
    \end{split}
    \end{equation}
    for all $t\in(0,T_{\max,\varepsilon})$.  Let $c_1:=\frac{q}{2(q+\sqrt{3}+1)^2}$ and $c_2:=\frac{q}{2(q+\sqrt{3})^2}$, there exists $c_3>0$ fulfilling
    \begin{equation}\label{2.19}
        \begin{split}
            &\frac{q}{2}\int_{\partial\Omega}v_\varepsilon^{-q+1}|\nabla v_\varepsilon|^{q-2}\cdot\frac{\partial|\nabla v_\varepsilon|^2}{\partial\nu}\\
            &\leq\frac{c_1}{2}\int_\Omega v_\varepsilon^{-q+1}|\nabla v_\varepsilon|^{q-2}|D^2 v_\varepsilon|^2+\frac{c_2}{2}\int_\Omega v_\varepsilon^{-q-1}|\nabla v_\varepsilon|^{q+2}+c_3\int_\Omega v_\varepsilon.
        \end{split}
    \end{equation}
    Furthermore, by two well-known inequalities (Lemma 3.4 of \cite{WDCDSB}), we have
    \begin{equation}\label{2.20}
        q\int_\Omega v_\varepsilon^{-q+3}|\nabla v_\varepsilon|^{q-2}|D^2\ln v_\varepsilon|^2\geq c_1\int_\Omega v_\varepsilon^{-q+1}|\nabla v_\varepsilon|^{q-2}|D^2 v_\varepsilon|^2+c_2\int_\Omega v_\varepsilon^{-q-1}|\nabla v_\varepsilon|^{q+2}.
    \end{equation}
    Substituting the results from \eqref{2.17}--\eqref{2.20} into \eqref{2.16}, we obtain that
    \begin{align}\label{2.21}
         &\frac{d}{dt}\int_\Omega v_\varepsilon^{-q+1}|\nabla v_\varepsilon|^q+c_1\int_\Omega v_\varepsilon^{-q+1}|\nabla v_\varepsilon|^{q-2}|D^2 v_\varepsilon|^2+c_2\int_\Omega v_\varepsilon^{-q-1}|\nabla v_\varepsilon|^{q+2}\\
         &\leq c_4\int_\Omega w_\varepsilon v_\varepsilon^{-q+2}|\nabla v_\varepsilon|^{q-2}|D^2 v_\varepsilon|+\frac{c_1}{2}\int_\Omega v_\varepsilon^{-q+1}|\nabla v_\varepsilon|^{q-2}|D^2 v_\varepsilon|^2+\frac{c_2}{2}\int_\Omega v_\varepsilon^{-q-1}|\nabla v_\varepsilon|^{q+2}+c_3\int_\Omega v_\varepsilon\nonumber
    \end{align}
    with $c_4:=q(q+\sqrt{3}-2)$. Therefore, an application of Young's inequality shows that for all $t\in(0,T_{\max,\varepsilon})$,
    \begin{equation*}
        \begin{split}
            &c_4\int_\Omega w_\varepsilon v_\varepsilon^{-q+2}|\nabla v_\varepsilon|^{q-2}|D^2v_\varepsilon|\\
            &\leq\frac{c_1}{2}\int_\Omega v_\varepsilon^{-q+1}|\nabla v_\varepsilon|^{q-2}|D^2v_\varepsilon|^2+\frac{c_5^2}{2c_1}\int_\Omega w_\varepsilon^2v_\varepsilon^{-q+3}|\nabla v_\varepsilon|^{q-2}\\
            &=\frac{c_1}{2}\int_\Omega v_\varepsilon^{-q+1}|\nabla v_\varepsilon|^{q-2}|D^2v_\varepsilon|^2
            +\int_\Omega\bigg\{\frac{c_2}{2}v_\varepsilon^{-q-1}|\nabla v_\varepsilon|^{q+2}\bigg\}^{\frac{q-2}{q+2}}\cdot
            \bigg\{\bigg(\frac{2}{c_2}\bigg)^{\frac{q-2}{q+2}}\cdot\frac{c_4^2}{2c_1}w_\varepsilon^2v_\varepsilon^{\frac{4}{q+2}}\bigg\}\\
            &\leq\frac{c_1}{2}\int_\Omega v_\varepsilon^{-q+1}|\nabla v_\varepsilon|^{q-2}|D^2v_\varepsilon|^2+\frac{c_2}{2}\int_\Omega v_\varepsilon^{-q-1}|\nabla v_\varepsilon|^{q+2}+\bigg(\frac{2}{c_2}\bigg)^{\frac{q-2}{4}}
            \cdot\bigg(\frac{c_4^2}{2c_1}\bigg)^{\frac{q+2}{4}}\int_{\Omega}w_\varepsilon^{\frac{q+2}{2}}v_\varepsilon.
        \end{split}
    \end{equation*}
     Thus, due to \eqref{2.21}, we obtain that
    \begin{equation*}
        \begin{split}
            &\frac{d}{dt}\int_\Omega v_\varepsilon^{-q+1}|\nabla v_\varepsilon|^q+\gamma(q)\int_\Omega v_\varepsilon^{-q-1}|\nabla v_\varepsilon|^{q+2}\leq C(q)\bigg(\int_{\Omega}w_\varepsilon^{\frac{q+2}{2}}v_\varepsilon+\int_\Omega v_\varepsilon\bigg),
        \end{split}
    \end{equation*}
    where $\gamma(q):=\frac{c_2}{2}=\frac{q}{4(q+\sqrt{3})^2}$. The proof of this lemma is complete.
\end{proof}
In view of the results of Lemma \ref{lemma2.5} on the boundedness of $\frac{1}{v_\varepsilon(\cdot,t)}$, we can derive the local boundedness of $\|u_\varepsilon(\cdot,t)\|_{L^p(\Omega)}$.
\begin{lemma}\label{lemma2.8}
    For all $p\geq2$ and each $T>0$, there exists $C(p,T)>0$ such that
    \begin{equation}\label{2.25}
        \int_\Omega u_\varepsilon^p(\cdot,t)\leq C(p,T)~~\mbox{for all}~~ t\in(0,T)\cap(0,T_{\max,\varepsilon}),
    \end{equation}
    and
    \begin{equation}\label{2.26}
    \int_0^{t}\int_\Omega u_\varepsilon^{p-1}v_\varepsilon|\nabla u_\varepsilon|^2\leq C(p,T)~~\mbox{for all}~~ t\in(0,T)\cap(0,T_{\max,\varepsilon}).
    \end{equation}
\end{lemma}

\begin{proof}
    Multiplying the first equation in \eqref{2.5} by $u_\varepsilon^{p-1}$ and using Young's inequality, we have
     \begin{align}\label{2.27}
           \frac{1}{p}\frac{d}{dt}\int_\Omega u_\varepsilon^p&=\int_\Omega u_\varepsilon^{p-1}\{\nabla\cdot(u_\varepsilon v_\varepsilon\nabla u_\varepsilon)-\nabla\cdot(u_\varepsilon^2v_\varepsilon\nabla v_\varepsilon)+\ell v_\varepsilon w_\varepsilon\}\nonumber\\
           &=-(p-1)\int_\Omega u_\varepsilon^{p-1}v_\varepsilon|\nabla u_\varepsilon|^2+(p-1)\int_\Omega u_\varepsilon^pv_\varepsilon\nabla u_\varepsilon\cdot\nabla v+\ell\int_\Omega u_\varepsilon^{p-1}v_\varepsilon w_\varepsilon\\
           &\leq-\frac{p-1}{2}\int_\Omega u_\varepsilon^{p-1}v_\varepsilon|\nabla u_\varepsilon|^2+\frac{(p-1)}{2}\int_\Omega u_\varepsilon^{p+1}v_\varepsilon|\nabla v_\varepsilon|^2+\frac{\ell(p-1)}{p}\int_\Omega u_\varepsilon^pv_\varepsilon+\frac{\ell}{p}\int_\Omega v_\varepsilon w_\varepsilon^{p}\nonumber
    \end{align}
    for all $t\in(0,T)\cap(0,T_{\max,\varepsilon})$. Due to the inequality $(a-b)^2\geq\frac{1}{2}a^2-b^2$ for all $a,b\in\mathbb{R}$, we derive that
    \begin{equation}\label{2.30a}
        \begin{split}
            \int_\Omega u_\varepsilon^{p-1}v_\varepsilon|\nabla u_\varepsilon|^2&=\frac{4}{(p+1)^2}\int_\Omega|v_\varepsilon^{\frac{1}{2}}\nabla u_\varepsilon^{\frac{p+1}{2}}|^2
            \\
            &=\frac{4}{(p+1)^2}\int_\Omega|\nabla(u_\varepsilon^{\frac{p+1}{2}}v_\varepsilon^{\frac{1}{2}})-\frac{1}{2}u_\varepsilon^{\frac{p+1}{2}}v_\varepsilon^{-\frac{1}{2}}\nabla v_\varepsilon|^2\\
            &\geq\frac{2}{(p+1)^2}\int_\Omega|\nabla(u_\varepsilon^{\frac{p+1}{2}}v_\varepsilon^{\frac{1}{2}})|^2-\frac{1}{(p+1)^2}\int_\Omega u_\varepsilon^{p+1}v_\varepsilon^{-1}|\nabla v_\varepsilon|^2.
        \end{split}
    \end{equation}
   Combining  \eqref{2.30a} with  \eqref{2.27} yields
      \begin{align}\label{2.31a}
           &\frac{d}{dt}\int_\Omega u_\varepsilon^p
           + \frac{p(p-1)}{4}\int_\Omega u_\varepsilon^{p-1}v_\varepsilon|\nabla u_\varepsilon|^2
           +\frac1{10}\int_\Omega|\nabla(u_\varepsilon^{\frac{p+1}{2}}v_\varepsilon^{\frac{1}{2}})|^2           \\
           \leq & \frac 14 \int_\Omega u_\varepsilon^{p+1}v_\varepsilon^{-1}|\nabla v_\varepsilon|^2 +\frac{p(p-1)}{2}\int_\Omega u_\varepsilon^{p+1}v_\varepsilon|\nabla v_\varepsilon|^2+\ell(p-1)\int_\Omega u_\varepsilon^pv_\varepsilon+\ell\int_\Omega v_\varepsilon w_\varepsilon^{p}\nonumber
    \end{align}
    By H\"{o}lder's inequality, we get
    \begin{equation}\label{2.29}
    \begin{split}
        \int_\Omega u_\varepsilon^{p+1}v_\varepsilon^{-1}|\nabla v_\varepsilon|^2&\leq\bigg\{\int_\Omega u_\varepsilon^{2(p+1)}v_\varepsilon^{2}\bigg\}^{\frac{1}{2}}\cdot\bigg\{\int_\Omega\frac{|\nabla v_\varepsilon|^4}{v_\varepsilon^4}\bigg\}^{\frac{1}{2}}\\
    \end{split}
    \end{equation}
    and
    \begin{equation}\label{2.30}
    \begin{split}
        \int_\Omega u_\varepsilon^{p+1}v_\varepsilon|\nabla v_\varepsilon|^2&\leq\bigg\{\int_\Omega u_\varepsilon^{2(p+1)}v_\varepsilon^2\bigg\}^{\frac{1}{2}}\cdot\bigg\{\int_\Omega|\nabla v_\varepsilon|^4\bigg\}^{\frac{1}{2}}.\\
    \end{split}
    \end{equation}
    In addition, according to Lemma \ref{lemma2.5} and Lemma \ref{lemma2.6},  there exists $c_1=c_1(T)>0$ such that
    \begin{equation}\label{2.31}
        \|\frac{1}{v_\varepsilon}(\cdot,t)\|_{L^\infty(\Omega)} +\|\nabla v_\varepsilon(\cdot,t)\|_{L^4(\Omega)}^4\leq c_1~~\mbox{for all}~~t\in(0,T)\cap(0,T_{\max,\varepsilon}).
    \end{equation}
    Therefore from \eqref{2.31a}-\eqref{2.31},
    it follows that
     \begin{align}\label{2.35a}
           &\frac{d}{dt}\int_\Omega u_\varepsilon^p
           + \frac{p(p-1)}{4}\int_\Omega u_\varepsilon^{p-1}v_\varepsilon|\nabla u_\varepsilon|^2
           +\frac1{10}\int_\Omega|\nabla(u_\varepsilon^{\frac{p+1}{2}}v_\varepsilon^{\frac{1}{2}})|^2           \\
           \leq & c_2(p,T)\bigg\{\int_\Omega u_\varepsilon^{2(p+1)}v_\varepsilon^2\bigg\}^{\frac{1}{2}}
           +\ell(p-1)\int_\Omega u_\varepsilon^pv_\varepsilon+\ell\int_\Omega v_\varepsilon w_\varepsilon^{p}.\nonumber
    \end{align}
     On the other hand, along with \eqref{2.7} and \eqref{2.8}, the Gagliardo-Nirenberg inequality provides $c_3>0$ such that
        \begin{equation}\label{2.32}
    \begin{split}
        \bigg\{\int_\Omega u_\varepsilon^{2(p+1)}v_\varepsilon^2\bigg\}^{\frac1 2}&=\|u_\varepsilon^{\frac{p+1}{2}}v_\varepsilon^{\frac{1}{2}}\|_{L^4(\Omega)}^2\\
        &\leq c_3\|\nabla(u_\varepsilon^{\frac{p+1}{2}}
        v_\varepsilon^{\frac{1}{2}})\|_{L^2(\Omega)}^{2a}
        \|u_\varepsilon^{\frac{p+1}{2}}
        v_\varepsilon^{\frac{1}{2}}\|_{L^{\frac{2}{p+1}}(\Omega)}^{2-2a}
        +c_3\|u_\varepsilon^{\frac{p+1}{2}}v_\varepsilon^{\frac{1}{2}}\|_{L^\frac{2}{p+1}(\Omega)}^2\\
        &\leq c_3\|v_0\|_{L^\infty(\Omega)}^{1-a}m^{(p+1)(1-a)}\|\nabla(u_\varepsilon^{\frac{p+1}{2}}v_\varepsilon^{\frac{1}{2}})\|_{L^2(\Omega)}^{2a}
        +c_3\|v_0\|_{L^\infty(\Omega)}m^{p+1},\\
    \end{split}
    \end{equation}
    where $a:=\frac{6p+3}{6p+4}\in(0,1)$.
    Hence, by Young's inequality, the combination \eqref{2.35a} with \eqref{2.32} leads to
         \begin{align}\label{2.37a}
           &\frac{d}{dt}\int_\Omega u_\varepsilon^p
           + \frac{p(p-1)}{4}\int_\Omega u_\varepsilon^{p-1}v_\varepsilon|\nabla u_\varepsilon|^2
           +\frac1{20}\int_\Omega|\nabla(u_\varepsilon^{\frac{p+1}{2}}v_\varepsilon^{\frac{1}{2}})|^2           \\
           \leq & c_4(p,T)+\ell(p-1)\int_\Omega u_\varepsilon^pv_\varepsilon+\ell\int_\Omega v_\varepsilon w_\varepsilon^{p}\nonumber
    \end{align}
     with some $c_4(p,T)>0$   for all $ t\in(0,T)\cap(0,T_{\max,\varepsilon})
$.
At this position, taking $p=2$ in  \eqref{2.37a} and applying inequality \eqref{2.14} to $\mu:=2, p:=1$, one can find
$c_5(T)>0$ such that
 \begin{equation*}
 \frac{d}{dt}\int_\Omega u_\varepsilon^2+\frac1 {20}\int_\Omega|\nabla(u_\varepsilon^{\frac{3}{2}}v_\varepsilon^{\frac{1}{2}})|^2\leq \ell\|v_0\|_{L^\infty(\Omega)}\int_\Omega u_\varepsilon^2+c_5(T),
 \end{equation*}
  which implies that for some $c_6(T)>0$,
  \begin{equation*}
  \int_\Omega u_\varepsilon^2(\cdot,t)\leq c_6(T)~~\mbox{for all}~~t\in(0,T)\cap(0,T_{\max,\varepsilon}).
  \end{equation*}
  As the application of Lemma \ref{lemma2.4} to $p:=2$, this  leads to
  $\|w(\cdot,t)\|_{L^\infty(\Omega)}\leq c_7(p,T)$ with  constant  $c_7(T)>0$ for all $t\in(0,T)\cap(0,T_{\max,\varepsilon})$.

   Furthermore, it follows from \eqref{2.37a} that for some $c_8(p,T)>0$, we have
  \begin{equation*}
  \frac{d}{dt}\int_\Omega u_\varepsilon^p+\frac{p(p-1)}{4}\int_\Omega u_\varepsilon^{p-1}v_\varepsilon|\nabla u_\varepsilon|^2
           \leq \ell \|v_0\|_{L^\infty(\Omega)}p \int_\Omega u_\varepsilon^p+ c_8(p,T),
  \end{equation*}
  and thereby derive \eqref{2.25} and \eqref{2.26} by the Gronwall inequality.
\end{proof}

\begin{lemma}\label{lemma2.9} For all $T>0$, there exists  $C(T)>0$ such that
\begin{align}\label{2.35}
   \|v_{\varepsilon}(\cdot,t)\|_{W^{1,\infty}(\Omega)}\leq C(T)\qquad for\; all\; t\in (0,T)\cap(0,T_{\max,\varepsilon}).
\end{align}
\end{lemma}
\begin{proof}
    Based on Lemma \ref{lemma2.8}, we obtain the boundedness of $u_{\varepsilon}$ in $L^{\infty}((0,\min\{T,T_{max,\varepsilon}\});L^2(\Omega))$. As an application of Lemma \ref{lemma2.4} to $p:=2$,  the boundedness of $w_{\varepsilon}$ in $L^{\infty}((0,\min\{T,T_{max,\varepsilon}\});L^\infty(\Omega))$ is achieved
    immediately. This result together with Lemma \ref{lemma2.6}, yields \eqref{2.35} readily.
\end{proof}

At this position, we can  proceed to prove  the $L^\infty$-bounds for $u_{\varepsilon}$.
\begin{lemma}\label{lemma2.10}
For all $T>0$,  there exists
$C(T)>0$
 such that
\begin{align}\label{2.36}
   \|u_{\varepsilon}(\cdot,t)\|_{L^{\infty}(\Omega)}\leq C( T) ~\qquad for\; all\; ~t\in (0,T)\cap(0,T_{max,\varepsilon}).
\end{align}
\end{lemma}
\begin{proof}
We begin by reformulating the governing equation for $u_\varepsilon$ as
$$u_{\varepsilon t}=\nabla\cdot\left(A_{\varepsilon}(x,t,u_{\varepsilon})\nabla u_{\varepsilon}\right)+\nabla\cdot B_{\varepsilon}(x,t)+D_{\varepsilon}(x,t) \qquad(x,t)\in\Omega\times(0,T_{max,\varepsilon}),$$
where the nonlinear operators are defined by
$$A_{\varepsilon}(x,t,\xi):=v_{\varepsilon}(x,t)\xi \qquad(x,t,\xi)\in\Omega\times (0,T_{max,\varepsilon})\times[0,\infty),$$
and
$$B_{\varepsilon}(x,t):=-\chi u_{\varepsilon}^{2}(x,t)v_{\varepsilon}(x,t)\nabla v_{\varepsilon}(x,t)\qquad(x,t)\in\Omega\times(0,T_{max,\varepsilon}),$$
as well as
$$ D_{\varepsilon}(x,t):=\ell v_{\varepsilon}(x,t)w_{\varepsilon}(x,t) \qquad(x,t)\in\Omega\times(0,T_{max,\varepsilon}).$$
Through applications of Lemma \ref{lemma2.5}, Lemma \ref{lemma2.8} and Lemma \ref{lemma2.9}, we derive the following estimates for any $p > 1$
$$A_{\varepsilon}(x,t,\xi)\geq c_{1}(T)\xi\qquad\text{for all }(x,t,\xi)\in\Omega\times \big(0,\min\{T,T_{\max,\varepsilon}\}\big)\times[0,\infty),$$
and
$$\|B_{\varepsilon}(\cdot,t)\|_{L^{p}(\Omega)}\leq c_{2}\quad\text{and}\quad\|D_{\varepsilon}(\cdot,t)\|_{L^{p}(\Omega)}\leq c_2\qquad\text{for all }t\in(0,T)\cap(0,T_{max,\varepsilon}),$$
with some $c_1(T)>0,c_2>0$. Thereafter  \eqref{2.36}  can be derived from  a Moser-type iteration, as recorded in (\cite{TaoW}, Lemma A.1).
\end{proof}

As an application of \eqref{2.6} and Lemma \ref{lemma2.10}, one can show that the solutions to \eqref{2.5} are in fact global in time.
\begin{lemma}\label{lemma2.11}Under the assumptions of Theorem 1.1, we  have
$T_{max,\varepsilon}=\infty$ for all $\varepsilon\in(0,1).$
\end{lemma}
Drawing on the a priori estimates from the previous lemmas, we now deduce the
 H\"{o}lder regularity for the global solution to \eqref{2.5} by applying standard parabolic regularity theory.
\begin{lemma}\label{lemma2.12}
For all $T>0$ and
$\varepsilon\in(0,1)$, there exist $\theta_1=\theta_1(T)\in(0,1)$ and $C(T)>0$ such that
\begin{align}\label{2.37}
   \|u_{\varepsilon}\|_{C^{\theta_1,\frac{\theta_1}{2}}(\overline{\Omega}\times[0,T])} \leq C(T),
\end{align}
\begin{align}\label{2.38}
   \|w_{\varepsilon}\|_{C^{\theta_1,\frac{\theta_1}{2}}(\overline{\Omega}\times[0,T])} \leq C(T),
\end{align}
and
\begin{align}\label{2.39}
    \|v_{\varepsilon}\|_{C^{\theta_1,\frac{\theta_1}{2}}(\overline{\Omega}\times[0,T])} \leq C(T).
\end{align}
Moreover, for all $\tau>0$ and $T>\tau$ there exist $\theta_2=\theta_2(\tau,T)\in(0,1)$ and $C(\tau,T)>0$ such that
\begin{align}\label{2.40}
    \left\|u_{\varepsilon}\right\|_{C^{2+\theta_2,1+\frac{\theta_2}{2}}(\overline{\Omega}\times[\tau,T])}+
    \left\|v_{\varepsilon}\right\|_{C^{2+\theta_2,1+\frac{\theta_2}{2}}(\overline{\Omega}\times[\tau,T])}
    +\left\|w_{\varepsilon}\right\|_{C^{2+\theta_2,1+\frac{\theta_2}{2}}(\overline{\Omega}\times[\tau,T])}\leq C(\tau,T)
\end{align}
for all $\varepsilon\in(0,1)$.
\end{lemma}
\begin{proof}
By employing the uniform $L^{\infty}$ bounds for $(u_\varepsilon)_{\varepsilon\in(0,1)}$, $(v_\varepsilon)_{\varepsilon\in(0,1)}$,
$(\nabla v_\varepsilon)_{\varepsilon\in(0,1)}$ and $(w_\varepsilon)_{\varepsilon\in(0,1)}$ established in Lemma \ref{lemma2.10}, \eqref{2.7}, Lemma \ref{lemma2.9} and Lemma \ref{lemma2.4} to control the source terms, we may directly deduce \eqref{2.37} through the standard parabolic H\"{o}lder regularity theory \cite{Holder}. The verification of \eqref{2.38} and \eqref{2.39} can be achieved through a simplified variant of this argument. The estimate \eqref{2.40} is then an immediate consequence of classical parabolic Schauder theory \cite{Schauder}, applied in conjunction with the H\"{o}lder regularity established in \eqref{2.37}-\eqref{2.39}.
\end{proof}

\begin{lemma}\label{lemma2.15}
    Suppose that $(u_0,v_0,w_0)$ satisfies \eqref{1.2} and \eqref{1.6}. Then there exists
    $(\varepsilon_{j})_{j\in \mathbb{N}}\subset (0,1)$ and functions $(u, v, w)$ fulfilling \eqref{1.7}, as well as $u\geq 0, v>0$ and $w>0$ in $\overline{\Omega}\times(0,\infty)$,
such that
\begin{align}
    &u_{\varepsilon}\rightarrow{u}\qquad \text{in}\quad C_{loc}^{0}(\overline{\Omega}\times[0,\infty)),\label{2.41}\\
    &v_{\varepsilon}\rightarrow{v}\qquad \text{in}\quad C_{loc}^{0}(\overline{\Omega}\times[0,\infty))\cap C_{loc}^{2,1}(\overline{\Omega}\times(0,\infty)),\label{2.42}\\
    &w_{\varepsilon}\rightarrow{w}\qquad \text{in}\quad C_{loc}^{0}(\overline{\Omega}\times[0,\infty))\cap C_{loc}^{2,1}(\overline{\Omega}\times(0,\infty)),\label{2.43}\\
    &\nabla v_{\varepsilon}\stackrel{*}{\rightharpoonup}\nabla v\qquad  \text{in}\quad L_{loc}^{\infty}(\Omega\times(0,\infty)),\label{2.44}\\
    &\nabla w_{\varepsilon}\stackrel{*}{\rightharpoonup}\nabla w\qquad  \text{in}\quad L_{loc}^{\infty}(\Omega\times(0,\infty))\label{2.45}
\end{align}
as $\varepsilon=\varepsilon_{j}\searrow0$, and that $(u,v,w)$ forms a continuous global weak solutions of \eqref{1.1} in the sense of
Definition \ref{Definition 2.1}.
\end{lemma}
\begin{proof}
Applying Young's inequality, there exists $C=C(K,T)>0$ such that for all $T>0$,
    \begin{align*}
        \int_0^{T}\int_{\Omega}|\nabla u_{\varepsilon}^2|\leq \int_0^{T}\int_{\Omega}u_{\varepsilon}v_{\varepsilon}|\nabla u_{\varepsilon}|^2+\int_0^{T}\int_{\Omega}u_{\varepsilon}v_{\varepsilon}^{-1}\leq C,
    \end{align*}
    we may see that thanks to Lemma \ref{lemma2.8} and Lemma \ref{lemma2.9},
    $$(\nabla u_{\varepsilon}^2)_{\varepsilon\in(0,1)}~~ \mbox{is bounded in}~~L^1((0,T);W^{1,1}(\Omega))~~\mbox{for all}~~T>0.$$
Subsequently, employing Lemma \ref{lemma2.9}, Lemma \ref{lemma2.12} and the Arzel\'{a}-Ascoli compactness theorem, a straightforward extraction procedure allows us to construct a vanishing subsequence $(\varepsilon_{j})_{j\in \mathbb{N}}$ with $\varepsilon_{j}\to0$. This yields nonnegative functions $(u,v,w)$ satisfying \eqref{1.7}, \eqref{2.1} and \eqref{2.41}--\eqref{2.45}. The derivation of the identities in \eqref{2.2}, \eqref{2.3} and \eqref{2.4} can thereupon be accomplished on the basis of these convergence properties in a straightforward manner, taking $\varepsilon=\varepsilon_j\to0$ in the respective weak formulations associated with \eqref{2.5}.
\end{proof}
\section{Global Boundedness}
In this section, we shall establish the global boundedness for $\int_\Omega u_\varepsilon^{\frac3 2}$ through a self-mapping approach, which serves as the starting point for  deriving of  further regularity properties of the solutions.  In light of the degenerate-dissipation quantity $\int_{\Omega} u_\varepsilon^{k} v_\varepsilon |\nabla u_\varepsilon|^2$,
 the following lemma provides the necessary groundwork for Lemma \ref{lemma3.4} and Lemma \ref{lemma3.6}.  
\begin{lemma}\label{lemma3.1}
Let $\Omega \subset \mathbb R^3$ be  a smoothly bounded domain, and supposed that $p>0$ and $1< r<6$. Then for any $\eta>0$ one can find constant $C(\eta,p,r)>0$
such that
\begin{align}\label{3.1}
    \|\phi^{\frac{p+1}{2}}\sqrt{\psi}\|^2_{L^{r}(\Omega)}\leq \eta \int_{\Omega}\phi^{p-1}\psi|\nabla \phi|^2+\eta
    \int_{\Omega}\phi^{p+1}\psi^{-1}|\nabla \psi|^2+C(\eta, p,r)\cdot
    \left\{\int_{\Omega}\phi\right\}^{p}\left\{\int_{\Omega}\phi\psi\right\}
\end{align}for all $\phi\in C^1(\Bar{\Omega})$ and $\psi\in C^1(\overline{\Omega})$  with $\phi>0$ and $\psi>0$ in $\overline{\Omega}$.
\end{lemma}
\begin{proof} From the Gagliardo-Nirenberg inequality,   it follows that for $1<r<6$ there exists $c_1=c_1(r,\Omega)>0$ such that
    \begin{align*}
      \|\phi^{\frac{p+1}{2}}\sqrt{\psi}\|^2_{L^{r}(\Omega)}&\leq c_1\|\nabla
      (\phi^{\frac{p+1}{2}}\sqrt{\psi})\|^{2a}_{L^{2}(\Omega)}\|\phi^{\frac{p+1}{2}}\sqrt{\psi}\|^{2(1-a)}_{L^{1}(\Omega)}
      +c_1\|\phi^{\frac{p+1}{2}}\sqrt{\psi}\|_{L^{1}(\Omega)}^2,
          \end{align*}
          with $a=\frac{6(r-1)}{5r}$.
          For any $\eta>0$,
    we use Young's inequality to get
    \begin{align*}
      c_1\|\nabla
      (\phi^{\frac{p+1}{2}}\sqrt{\psi})\|^{2a}_{L^{2}(\Omega)}
      \|\phi^{\frac{p+1}{2}}\sqrt{\psi}\|^{2(1-a)}_{L^{1}(\Omega)}
      &\leq \eta_1 \|\nabla (\phi^{\frac{p+1}{2}}\sqrt{\psi})\|_{L^{2}(\Omega)}^2
      +c_2(\eta_1,a)\|\phi^{\frac{p+1}{2}}\sqrt{\psi}\|_{L^{1}(\Omega)}^2,
    \end{align*}
    where
          \begin{equation}\label{3.1*}
          \eta_1\equiv\eta_1(\eta,p):=\frac{\eta}{(p+1)^2},\quad c_2(\eta_1,a)={\eta_1}^{\frac{a}{a-1}} {c_1}^{\frac{1}{1-a}}.
          \end{equation}
    Now if $p\geq 1$, by an interpolation inequality for $L^1(\Omega)$ between $L^r(\Omega)$ and $L^{\frac{2}{p+1}}(\Omega)$, we then have
         \begin{equation*}
         \big(c_1+c_2(\eta_1,a)\big)\|\phi^{\frac{p+1}{2}}\sqrt{\psi}\|_{L^{1}(\Omega)}^2\leq
    \frac{1}{2}\|\phi^{\frac{p+1}{2}}\sqrt{\psi}\|^2_{L^{r}(\Omega)}
    +c_3(\eta_1,p,r)\|\phi^{\frac{p+1}{2}}\sqrt{\psi}\|^2_{L^{\frac 2{p+1}}(\Omega)}\end{equation*}
    for some $c_3(\eta_1,p,r)>0$. 
      Therefore for all  $p\geq 1$, we arrive at
         \begin{equation}\label{3.2b}
         \|\phi^{\frac{p+1}{2}}\sqrt{\psi}\|^2_{L^{r}(\Omega)}\leq 2\eta_1 \|\nabla
    (\phi^{\frac{p+1}{2}}\sqrt{\psi})\|_{L^{2}(\Omega)}^2+2c_3(\eta_1,p,r)\|\phi^{\frac{p+1}{2}}\sqrt{\psi}\|^2_{L^{\frac{2}
    {p+1}}(\Omega)}.
    \end{equation}
   On the other hand, for $p\in(0,1)$, we use the H\"{o}lder inequality to  see  that
    \begin{align}\label{3.2c}
      \|\phi^{\frac{p+1}{2}}\sqrt{\psi}\|^2_{L^{r}(\Omega)}
      &\leq \eta_1 \|\nabla
      (\phi^{\frac{p+1}{2}}\sqrt{\psi})\|_{L^{2}(\Omega)}^2+(c_1+c_2(\eta_1,p,r))\|\phi^{\frac{p+1}{2}}\sqrt{\psi}\|_{L^{1}(\Omega)}^2
     \nonumber \\
      &\leq \eta_1 \|\nabla
      (\phi^{\frac{p+1}{2}}\sqrt{\psi})\|_{L^{2}(\Omega)}^2+c_4(\eta_1,p,r)\|\phi^{\frac{p+1}{2}}\sqrt{\psi}\|_{L^{\frac2{p+1}}(\Omega)}^2.
    \end{align}
    At this position,  an application  of the H\"{o}lder inequality yields
    \begin{equation}\label{3.2d}
    \begin{split}
    \|\phi^{\frac{p+1}{2}}\sqrt{\psi}\|_{L^{\frac2{p+1}}(\Omega)}^2&=\bigg\{\int_\Omega(\phi\psi)^{\frac1{p+1}}\phi^{\frac{p}{p+1}}\bigg\}^{p+1}\\
    &\leq\bigg\{\int_\Omega\phi\bigg\}^{p}\cdot\bigg\{\int_\Omega \phi\psi\bigg\}.
    \end{split}
    \end{equation}
    According to \eqref{3.1*}, a simple calculation shows that
    \begin{align}\label{3.2e}
    2\eta_1\|\nabla
    (\phi^{\frac{p+1}{2}}\sqrt{\psi})\|_{L^{2}(\Omega)}^2
    &=2\eta_1\int_\Omega\bigg|\frac{p+1}{2}\phi^{\frac{p-1}{2}}\psi^{\frac1 2}\nabla\phi+\frac1 2 \phi^{\frac{p+1}{2}}\psi^{-\frac1 2}\nabla\psi\bigg|^2\nonumber\\
    &\leq (p+1)^2\eta_1\int_\Omega\phi^{p-1}\psi|\nabla\phi|^2
    +\eta_1 \int_\Omega\phi^{p+1}\psi^{-1}|\nabla\psi|^2\nonumber\\
    &\leq\eta\int_\Omega\phi^{p-1}\psi|\nabla\phi|^2
    +\eta\int_\Omega\phi^{p+1}\psi^{-1}|\nabla\psi|^2.
    \end{align}
    Hence,  from \eqref{3.2b}--\eqref{3.2e}, it follows that \eqref{3.1} holds with the desired $C(\eta,p,r)>0$. 

\end{proof}

As an elementary but crucial preparation for our self-map type argument, employing the Neumann heat semigroup estimates similar to the proof of Lemma \ref{lemma2.4},  $\|\nabla w_\varepsilon(\cdot,t)\|_{L^{\theta}(\Omega)}$ can be controlled by $\|u_\varepsilon(\cdot,t)\|_{L^{\frac3 2}(\Omega)}$.
\begin{lemma}\label{lemma3.2}
For any $\theta<3$, then there exists $C(\theta)>0$ such that
\begin{equation}\label{3.3}
    \|\nabla w_\varepsilon(\cdot,t)\|_{L^{\theta}(\Omega)}\leq C(\theta)\cdot(1+ \sup\limits_{t>0}\|u_\varepsilon(\cdot,t)\|_{L^{\frac{3}{2}}(\Omega)})~~\mbox{for all}~~t>0.
\end{equation}
\end{lemma}
\begin{proof}
By means of the variation-of-constants formula for $w_\varepsilon$, we have
	\begin{equation*}
		w_\varepsilon(\cdot,t)=e^{t(\Delta-1)}w_{0}+\int_{0}^{t}e^{(t-s)(\Delta-1)}u_\varepsilon(\cdot,s)ds.
	\end{equation*}
By the Neumann heat semigroup estimates in \cite{Winkler(2010)}, we see that
\begin{equation*}
\begin{split}
\|\nabla w_\varepsilon(\cdot,t)\|_{L^\theta(\Omega)}&=\big\|\nabla e^{t(\Delta-1)}w_0-\int_0^t\nabla e^{(t-s)(\Delta-1)}u_\varepsilon(\cdot,s)\big\|_{L^\theta(\Omega)}ds\\
&\leq c_1\|\nabla w_0\|_{L^\infty(\Omega)}+c_1\int_0^t(1+(t-s)^{-\frac1 2-\frac3 2(\frac2 3-\frac1 \theta)})e^{-\lambda(t-s)}\|u_\varepsilon(\cdot,s)\|_{L^{\frac3 2}(\Omega)}ds\\
&\leq c_2+c_2\sup\limits_{t>0}\|u_\varepsilon(\cdot,t)\|_{L^{\frac3 2}(\Omega)}\cdot \int_0^\infty(1+\sigma^{-\frac3 2+\frac3 {2\theta}})e^{-\lambda\sigma}d\sigma\\
&\leq c_3(\theta)(1+\sup\limits_{t>0}\|u_\varepsilon(\cdot,t)\|_{L^{\frac3 2}(\Omega)})
\end{split}
\end{equation*}
where due to our assumption $\theta<3$ implies that $\int_0^{\infty}(1+\sigma^{-\frac3 2+\frac3 {2\theta}})e^{-\lambda\sigma}d\sigma$ is finite. This yields \eqref{3.3}.
\end{proof}
As the preparation to make appropriate use of the dissipated contributions,  as quantified  through weighted expression of  $\int_{\Omega}u_\varepsilon^{k}v_\varepsilon|\nabla u_\varepsilon|^2$ within some range of $k$,
we  estimate $\int_0^{t}\int_\Omega u_\varepsilon^{-\frac3 4}v_\varepsilon|\nabla u_\varepsilon|^2$
 by $ \sup\limits_{t>0}\int_\Omega u_\varepsilon^{\frac3 2}(\cdot,t)$.
\begin{lemma}\label{lemma3.3}
Suppose that $K>0$ with the property \eqref{1.6} be valid. Then there exists $C(K)>0$ such that
    \begin{equation}\label{3.4}
        \int_0^{t}\int_\Omega u_\varepsilon^{-\frac3 4}v_\varepsilon|\nabla u_\varepsilon|^2\leq C(K)\bigg\{1+\bigg(\sup\limits_{t>0}\int_\Omega u_\varepsilon^{\frac3 2}(\cdot,t)\bigg)^{\frac{5}{6}}\bigg\}
    \end{equation}
    for all $\varepsilon\in(0,1)$ and $t>0$.
\end{lemma}
\begin{proof}
    Testing the first equation in \eqref{2.5} by $u_\varepsilon^{-\alpha}$, we obtain
    \begin{equation*}
    \begin{split}
        \frac{1}{1-\alpha}\frac{d}{dt}\int_\Omega u_\varepsilon^{1-\alpha}-\alpha\int_\Omega u_\varepsilon^{-\alpha} v_\varepsilon|\nabla u_\varepsilon|^2
        &= -\alpha\int_\Omega u_\varepsilon^{1-\alpha}v_\varepsilon\nabla v_\varepsilon\cdot\nabla u_\varepsilon+\ell\int_\Omega u_\varepsilon^{-\alpha}v_\varepsilon w_\varepsilon,\\
        \end{split}
    \end{equation*}
    which, upon applying the Cauchy-Schwarz inequality, yields
    \begin{equation*}
    \begin{split}
        \frac{1}{1-\alpha}\frac{d}{dt}\int_\Omega u_\varepsilon^{1-\alpha}-\frac{\alpha}2\int_\Omega u_\varepsilon^{-\alpha} v_\varepsilon|\nabla u_\varepsilon|^2
        &\geq -\frac{\alpha}2\int_\Omega u_\varepsilon^{2-\alpha}v_\varepsilon|\nabla v_\varepsilon|^2+\ell\int_\Omega u_\varepsilon^{-\alpha}v_\varepsilon w_\varepsilon.\\
        \end{split}
    \end{equation*}
  Now, setting $\alpha=\frac{3}{4}$, we derive
    \begin{equation}\label{3.5}
    \begin{split}
         \frac{3}8\int_\Omega u_\varepsilon^{-\frac{3}{4}} v_\varepsilon|\nabla u_\varepsilon|^2
         &\leq 4\frac{d}{dt}\int_\Omega u_\varepsilon^{\frac1 4}+\frac{3}{8}\int_\Omega u_\varepsilon^{\frac5 4}v_\varepsilon|\nabla v_\varepsilon|^2-\ell\int_\Omega u_\varepsilon^{-\alpha}v_\varepsilon w_\varepsilon\\
         \end{split}
    \end{equation}
    for all $\varepsilon\in(0,1)$.
   Integrating \eqref{3.5} over time and applying Lemma \ref{lemma2.3}, Lemma \ref{lemma2.6}, there exists $c_1=c_1(K)>0$ such that
    \begin{equation*}
        \begin{split}
            \int_0^t\int_\Omega u_\varepsilon^{-\frac{3}{4}} v_\varepsilon|\nabla u_\varepsilon|^2&\leq \frac{32}{3}\int_\Omega u_\varepsilon^{\frac1 4}(\cdot,t)+\int_0^t\int_\Omega u_\varepsilon^{\frac5 4}v_\varepsilon|\nabla v_\varepsilon|^2\\
            &\leq\frac{32}{3}\int_\Omega u_\varepsilon^{\frac1 4}(\cdot,t)+e^{2\kappa}\|v_0\|_{L^\infty(\Omega)}\int_0^t e^{-\kappa s}\int_\Omega \big(u_\varepsilon^{\frac3 2}\big)^{\frac5 6}\cdot\big(|\nabla v_\varepsilon|^2\big)\\
            &\leq 12\int_\Omega u_\varepsilon(\cdot,t)+12|\Omega|+e^{2\kappa}\|v_0\|_{L^\infty(\Omega)}\int_0^t e^{-\kappa s}\bigg\{\int_\Omega u_\varepsilon^{\frac{3}{2}} \bigg\}^{\frac{5}{6}}\bigg\{\int_\Omega |\nabla v_\varepsilon|^{12}\bigg\}^{\frac{1}{6}}\\
            &\leq 12m+12|\Omega|+c_1(K)\bigg(\sup\limits_{t>0}\int_\Omega u_\varepsilon^{\frac{3}{2}}\bigg)^{\frac{5}{6}}\int_0^t e^{-\kappa s}\\
        \end{split}
    \end{equation*}
    for all $t\in(0,\infty)$. This implies that \eqref{3.4} holds.
\end{proof}
The following refine interpolation inequality of Lemma \ref{lemma3.1} establishes the domination of the integral
 $\int_\Omega \varphi^{\beta}\psi$
by the dissipated quantity   $\int_{\Omega}\varphi^{k}\psi|\nabla \varphi|^2$ 
and  singularly weighted integral  $\int_\Omega\frac{|\nabla \psi|^q}{\psi^{q-1}}$ for parameters $\beta,k,q$ within appropriate ranges.


\begin{lemma}\label{lemma3.4}
Let $\Omega\subset\mathbb{R}^3$ be  a smoothly bounded domain, $k>-1$ and $\beta\in (k+2,k+\frac{8}{3})$. Then for all $m^*>0$, one can find $C_1=C_1(k, \beta)>0$ and $C_2=C_2(m^*,k,\beta)>0$ with the property that whenever $\varphi\in C^1(\overline \Omega)$ and $\psi \in C^1(\overline \Omega)$ are positive 
in $\overline \Omega$ fulfilling  $\int_{\Omega}\varphi\leq m^*$,
  \begin{align}\label{3.6}
    \int_{\Omega}\varphi^{\beta}\psi\leq  2\int_{\Omega}\varphi^{k}\psi|\nabla \varphi|^2+C_1\int_\Omega\frac{|\nabla \psi|^{\frac{2\beta}{\beta-k-2}}}{\psi^{\frac{2\beta}{\beta-k-2}-1}}+C_2\int_{\Omega}\varphi \psi
\end{align} holds.
\end{lemma}
\begin{proof}
    Let
 $\theta:=\frac{1}{k+3-\beta}$ and $ \theta_*:=\frac{1}{\beta-k-2}.$
   Then 
       $ \theta\in (1,3)$
   can be warranted by $\beta \in (k+2,k+\frac{8}{3})$.
Consequently, by the H\"{o}lder inequality and applying  Lemma \ref{lemma3.1} to $p:=k+1$ and $r:=2\theta$,  we conclude that
there exists
$c_1=c_1(\beta,k,m)$ such that
\begin{align*}
    \int_{\Omega}\varphi^{\beta}\psi
    &=\int_{\Omega}(\varphi^{\frac{k+2}{2}}\psi^{\frac{1}{2}})^2\varphi^{\beta-k-2} \nonumber
    \\
    &\leq\|\varphi^{\frac{k+2}{2}}\psi^{\frac{1}{2}}\|^2_{L^{2\theta}(\Omega)}\cdot\|\varphi^{\beta-k-2}\|_{L^{\theta_*}(\Omega)}\nonumber
    \\
    &\leq (m^*)^{\frac{1}{\theta_*}}\|\varphi^{\frac{k+2}{2}}\psi^{\frac{1}{2}}\|^2_{L^{2\theta}(\Omega)}\nonumber
    \\
    &\leq \int_{\Omega}\varphi^{k}\psi|\nabla
    \varphi|^2+  \int_{\Omega}\varphi^{k+2}\frac{|\nabla
    \psi|^2}{\psi}+c_1\int_{\Omega}\varphi \psi.
\end{align*}
Due to $k+2<\beta$ and using the Young inequality, there exists $c_2=c_2(k,\beta)>0$ such that
    \begin{equation*}
        \begin{split}
            \int_\Omega \varphi^{k+2}\frac{|\nabla \psi|^2}{\psi}&=\int_\Omega(\varphi^{\beta}\psi)^{\frac{k+2}{\beta}}\cdot\frac{|\nabla \psi|^2}{\psi^{1+\frac{k+2}{\beta}}}\\
            &\leq \frac{1}{2}\int_\Omega \varphi^{\beta}\psi+c_2\int_\Omega\frac{|\nabla \psi|^{\frac{2\beta}{\beta-k-2}}}{\psi^{\frac{2\beta}{\beta-k-2}-1}}.\\
        \end{split}
    \end{equation*}
      Hence, we have
    \begin{equation*}
        \int_\Omega \varphi^\beta \psi\leq 2\int_{\Omega}\varphi^{k}\psi|\nabla \varphi|^2+2c_2\int_\Omega\frac{|\nabla \psi|^{\frac{2\beta}{\beta-k-2}}}{\psi^{\frac{2\beta}{\beta-k-2}-1}}+2c_1\int_{\Omega}\varphi \psi,
    \end{equation*}
which yields \eqref{3.6}.
\end{proof}

    Thanks to the functional inequality expressed in \eqref{3.6}, we derive  global bounds of $\int_\Omega u_\varepsilon^{\frac{3}{2}}$
by constructing a closed-loop estimate.
 This result provides the foundation for establishing $L^p$-bounds of $u_\varepsilon$ for all $p\geq2$.
\begin{lemma}\label{lemma3.5}
 Let $K>0$ be given  in \eqref{1.6}. Then one can find $M=M(K)$ such that
    \begin{equation}\label{3.7}
        \int_\Omega u_\varepsilon^{\frac{3}{2}}(\cdot,t)\leq M~~\mbox{for all}~~t\in(0,\infty).
        \end{equation}
\end{lemma}
\begin{proof}
For fixed $K>0$  given in \eqref{1.6}, 
let
$$T^*:=\sup\bigg\{T_0\in(0,\infty)|\int_\Omega u_\varepsilon^{\frac{3}{2}}(\cdot,t)<M~~\mbox{for all}~~t\in(0,T_0)\bigg\},$$
where constant $M>0$ will be specified below. It is observed that from \eqref{1.6} and the continuity of $u_\varepsilon$, $T^*$ is well defined.
Supposed that $T^*<\infty$,  then
$$ \int_\Omega u_\varepsilon^{\frac{3}{2}}(\cdot,T^*)=M.$$

Let
$    a:=\frac{165}{61}(\frac{k+2}{\beta}-\frac{3}{5})$ and $\iota:=\frac{2}{3}\cdot\frac{\beta a}{\beta-k-2}.
$
Then \begin{equation*}
   a\in(0,1)~~\mbox{and}~~ \iota<1
\end{equation*}
can be warranted by
\begin{equation*}
k\in(-\frac{1}{6},-\frac{5}{66}) ~~~\hbox{and}~~
\beta\in(\beta_1(k),\beta_2(k)),
\end{equation*}
where $\beta_1(k):=\frac{171}{127}(k+2)$,  $\beta_2(k):=k+\frac{8}{3}$. In particular,  for given  $k_0:=-\frac 4 {33}$ and $\beta_0:=\frac{\beta_1(k_0)+\beta_2(k_0)} 2$,  we then have $ \beta_0>\frac 52$,
 \begin{equation}\label{3.8}
a_0:=\frac{165}{61}(\frac{k_0+2}{\beta_0}-\frac{3}{5})\in(0,1)~~~\hbox{and}~~~
      \iota_0:=\frac{2}{3}\cdot\frac{\beta_0 a_0}{\beta_0-k_0-2}<1.
 \end{equation}

As an application of Lemma \ref{lemma3.4} to $k:=k_0$ and $\beta:=\beta_0$, we have
\begin{equation}\label{3.9}
 \int_\Omega u_\varepsilon^{\beta_0}v_\varepsilon\leq
 2\int_{\Omega}u_\varepsilon^{k_0}v_\varepsilon|\nabla u_\varepsilon|^2+c_1\int_\Omega\frac{|\nabla v_\varepsilon|^{\frac{2\beta_0}{\beta_0-k_0-2}}}{v_\varepsilon^{\frac{2\beta_0}{\beta_0-k_0-2}-1}}
 +c_2\int_{\Omega}u_\varepsilon v_\varepsilon
        \end{equation}
for constants  $c_1$ and $c_2>0$.

 Multiplying the first equation of \eqref{2.5} by $u_\varepsilon^{\frac{1}{2}}$ and using Young's inequality, we obtain that
\begin{equation}\label{3.10}
    \begin{split}
        \frac{2}{3}\frac{d}{dt}\int_\Omega u_\varepsilon^{\frac{3}{2}}(\cdot,t)
        &\leq\int_\Omega\nabla\cdot(u_\varepsilon v_\varepsilon\nabla u_\varepsilon)\cdot u_\varepsilon^{\frac{1}{2}}-\int_\Omega\nabla\cdot(u_\varepsilon^2v_\varepsilon\nabla v_\varepsilon)\cdot u_\varepsilon^{\frac{1}{2}}+\ell\int_\Omega u_\varepsilon ^{\frac{1}{2}}v_\varepsilon w_\varepsilon\\
        &\leq-\frac12\int_\Omega u_\varepsilon^{\frac{1}{2}}v_\varepsilon|\nabla u_\varepsilon|^2+\frac12\int_\Omega u_\varepsilon^{\frac{3}{2}}v_\varepsilon\nabla u_\varepsilon\cdot\nabla v_\varepsilon+\frac{\ell}{3}\int_\Omega u_\varepsilon ^{\frac{3}{2}}v_\varepsilon+\frac{2\ell}{3}\int_\Omega v_\varepsilon w_\varepsilon^{\frac3 2}\\
        &\leq-\frac{1}{4}\int_\Omega u_\varepsilon^{\frac{1}{2}}v_\varepsilon|\nabla u_\varepsilon|^2+\frac{1}{4}\int_\Omega u_\varepsilon^{\frac{5}{2}}v_\varepsilon|\nabla v_\varepsilon|^2+\frac{\ell}{3}\int_\Omega u_\varepsilon ^{\frac{3}{2}}v_\varepsilon+\frac{2\ell}{3}\int_\Omega v_\varepsilon w_\varepsilon^{\frac3 2}
        \end{split}
\end{equation}
for all $ t<T^*$. Therefore, thanks to \eqref{2.12}, \eqref{2.14} with $\mu=\frac3 2$ and \eqref{3.9},  it follows from \eqref{3.10} that for $ t\in (0,T^*)$,
\begin{align}\label{3.11}
        ~&\frac{8}{3}\frac{d}{dt}\int_\Omega u_\varepsilon^{\frac{3}{2}}(\cdot,t)+\int_\Omega u_\varepsilon^{\frac{1}{2}}v_\varepsilon|\nabla u_\varepsilon|^2\nonumber\\
        &\leq\int_\Omega u_\varepsilon^{\frac{5}{2}}v_\varepsilon|\nabla v_\varepsilon|^2+\frac{4\ell}{3}\int_\Omega u_\varepsilon ^{\frac{3}{2}}v_\varepsilon+\frac{8\ell}{3}\int_\Omega v_\varepsilon w_\varepsilon^{\frac3 2}\nonumber\\
        &= \int_\Omega (u_\varepsilon^{\beta_0}v_\varepsilon)^{\frac{5}{2\beta_0}}\cdot(v_\varepsilon^{1-\frac{5}{2\beta_0}}|\nabla v_\varepsilon|^{2})+\frac{4\ell}{3}\int_\Omega u_\varepsilon ^{\frac{3}{2}}v_\varepsilon+\frac{8\ell}{3}\int_\Omega v_\varepsilon w_\varepsilon^{\frac3 2}\nonumber\\
        &\leq \int_\Omega u_\varepsilon^{\beta_0}v_\varepsilon+c_3(\beta_0)\int_\Omega v_\varepsilon|\nabla v_\varepsilon|^{\frac{4\beta_0}{2\beta_0-5}}+c_4(\ell,\beta_0)\int_\Omega u_\varepsilon v_\varepsilon+\frac{8\ell}{3}\int_\Omega v_\varepsilon w_\varepsilon^{\frac3 2}\nonumber\\
        &\leq 2\int_{\Omega}u_\varepsilon^{k_0}v_\varepsilon|\nabla u_\varepsilon|^2+
        c_1\int_\Omega\frac{|\nabla v_\varepsilon|^{\frac{2\beta_0}{\beta_0-k_0-2}}}
        {v_\varepsilon^{\frac{2\beta_0}{\beta_0-k_0-2}-1}}
                 +c_4\int_{\Omega}u_\varepsilon v_\varepsilon\\
        &~+c_3(\beta_0)\int_\Omega v_\varepsilon|\nabla v_\varepsilon|^{\frac{4\beta_0}{2\beta_0-5}}+\frac{8\ell}{3}\int_\Omega v_\varepsilon w_\varepsilon^{\frac3 2}\nonumber\\
       &\leq 32 \int_\Omega u_\varepsilon^{-\frac{3}{4}}v_\varepsilon|\nabla u_\varepsilon|^2+
       \frac 12 \int_\Omega u_\varepsilon^{\frac{1}{2}}v_\varepsilon|\nabla u_\varepsilon|^2+
        c_1\int_\Omega\frac{|\nabla v_\varepsilon|^{\frac{2\beta_0}{\beta_0-k_0-2}}}
        {v_\varepsilon^{\frac{2\beta_0}{\beta_0-k_0-2}-1}}\nonumber\\
           &~+c_3(\beta_0)\int_\Omega v_\varepsilon|\nabla v_\varepsilon|^{\frac{4\beta_0}{2\beta_0-5}}+c_5e^{-\kappa t}\nonumber
\end{align}
for some $c_i>0, (i=2,3,4,5)$, due to $\beta_0>\frac 52$ and $-\frac 34<k_0<\frac 12$.

To estimate the third term in the right-side of \eqref{3.11}, we  apply \eqref{2.12}, Lemma \ref{lemma2.7} with $q:=\frac{2\beta_0}{\beta_0-k_0-2}-2$, the Gagliardo-Nirenberg inequality, along with \eqref{2.14} to obtain
    \begin{align}\label{3.12}
       & \frac{d}{dt}\int_{\Omega}
       \frac{
       |\nabla v_\varepsilon|^{
       \frac{2\beta_0}{\beta_0-k_0-2}-2}
       }
       {v_\varepsilon^{
       \frac{2\beta_0}{\beta_0-k_0-2}
       -3}}+
       \gamma \int_{\Omega}\frac{|\nabla v_\varepsilon|^{
       \frac{2\beta_0}{\beta_0-k_0-2}
       }}{v_\varepsilon^{
       \frac{2\beta_0}{\beta_0-k_0-2}
      -1}}\nonumber\\
        &\leq  c_6\bigg(\int_{\Omega}w_\varepsilon^{\frac{\beta_0}{\beta_0-k_0-2}}v_\varepsilon+\int_\Omega v_\varepsilon\bigg)\nonumber\\
        &\leq c_6 \|v_0\|_{L^{\infty}(\Omega)}e^{-\kappa t}\big(\int_\Omega w_\varepsilon^{\frac{\beta_0}{\beta_0-k_0-2}}+|\Omega|\big)\\
        &\leq c_7e^{-\kappa t}\|\nabla w_\varepsilon\|_{L^{\frac{11}{4}}(\Omega)}^{\frac{\beta_0 a_0}{\beta_0-k_0-2}}
        \|w_\varepsilon\|_{L^{\frac{5}{2}}(\Omega)}
        ^{\frac{\beta_0(1-a_0)}{\beta_0-k_0-2}}
        +c_7e^{-\kappa t}\|w_\varepsilon\|_{L^2(\Omega)}^{\frac{\beta_0}{\beta_0-k_0+2}}+c_7e^{-\kappa t}\nonumber\\
            &\leq c_8e^{-\kappa t}(1+M^{\frac{2}{3}})^{\frac{\beta_0 a_0}{\beta_0-k_0-2}}+c_8e^{-\kappa t}\nonumber\\
            &\leq c_9 (1+M^{\iota_0})e^{-\kappa t},\nonumber
    \end{align}
    here $q>2$ is warranted by $\beta_0<\beta_2(k_0)<2k_0+4$.
Therefore, combining \eqref{3.11}--\eqref{3.12}, we have
\begin{align}
&~\frac{d}{dt}\bigg(\frac{8\gamma}{3c_1}\int_\Omega u_\varepsilon^{\frac{3}{2}}(\cdot,t)+\frac{
       |\nabla v_\varepsilon|^{
       \frac{2\beta_0}{\beta_0-k_0-2}-2}}{v_\varepsilon^{
       \frac{2\beta_0}{\beta_0-k_0-2}
       -3}}\bigg)+\frac{\gamma}{2c_1}\int_\Omega u_\varepsilon^{\frac{1}{2}}v_\varepsilon|\nabla u_\varepsilon|^2\\
       &\leq \frac{32\gamma}{c_1} \int_\Omega u_\varepsilon^{-\frac{3}{4}}v_\varepsilon|\nabla u_\varepsilon|^2+\frac{c_3(\beta_0)\gamma}{c_1}\int_\Omega v_\varepsilon|\nabla v_\varepsilon|^{\frac{4\beta_0}{2\beta_0-5}}+\frac{c_5\gamma}{c_1}e^{-\kappa t}+c_9 (1+M^{\iota_0})e^{-\kappa t},\nonumber
\end{align}
which, upon integration and using \eqref{3.4} and Lemma \ref{2.6}, shows that
\begin{align}\label{3.14}
       \int_\Omega u_\varepsilon^{\frac{3}{2}}(\cdot,t)&\leq 12\int_0^t\int_{\Omega}u_\varepsilon^{-\frac{3}{4}}v_\varepsilon|\nabla u_\varepsilon|^2\nonumber+c_3(\beta_0)\int_0^t\int_\Omega v_\varepsilon|\nabla v_\varepsilon|^{\frac{4\beta_0}{2\beta_0-5}}+c_5\int_0^t e^{-\kappa t}\\
       &~~+c_9 (1+M^{\iota_0})\int_0^t e^{-\kappa t}+\int_\Omega u_0^{\frac{3}{2}}+\frac{3c_1}{8\gamma}
       \int_\Omega\frac{|\nabla v_0|^{\frac{2\beta_0}{\beta_0-k_0-2}-2}}
       {v_0^{\frac{2\beta_0}{\beta_0-k_0-2}-3}}\\
       &\leq c_{10}(K)(1+M^{\frac{5}{6}})+c_{11}(1+ M^{\iota_0})\int_0^{t}e^{-\kappa s}+c_{12}\int_0^{t}e^{-\kappa s}\nonumber\\
       &~~+\int_\Omega u_0^{\frac{3}{2}}
       +\frac{3c_1}{8\gamma}\int_\Omega\frac{|\nabla v_0|^{\frac{2\beta_0}{\beta_0-k_0-2}-2}}
       {v_0^{\frac{2\beta_0}{\beta_0-k_0-2}-3}}\nonumber\\
       &\leq C(K)(1+M^{\tau})\nonumber
\end{align}
where $\tau:=\max\{\frac{5}{6},\iota_0\}<1$.

Now at this position, we take  $M:=(2C(K)+1
)^{\frac{1}{1-\tau}}$, and then have
\begin{equation}\label{3.15}
    \int_\Omega u_\varepsilon^{\frac{3}{2}}(\cdot,t)\leq C(K)(1+ M^{\tau})\leq 2C(K)\cdot M^{\tau}\leq M-\frac 12
\end{equation}
for all $t\in(0,T^*)$. Due to the continuity of $u_\varepsilon$, \eqref{3.15} contradicts with the definition of $T^*$. Hence we have $T^*=\infty$.
\end{proof}

With Lemma \ref{lemma3.5} established, we now aim to derive bounds for $u_\varepsilon$ in all the  $L^p$ spaces. To this end,
we extend the exponent range of $\beta$ in Lemma \ref{lemma3.4} from $[k+2, k+\frac 8 3)$ to $[1, p+2)$,
under the assumption that a bound for  $\int_{\Omega}\varphi^{\frac{3}{2}}$ exists.
\begin{lemma}\label{lemma3.6}
Let $\Omega\subset\mathbb{R}^3$, $L>0$, $p>1$, $q>4$, and $\beta\in [1,p+2)$. Then for any $\eta\in(0,1)$, there exists $C=C(\eta,L,p,\beta
 )>0$ such that for all  $\varphi\in C^1(\overline \Omega)$ and $\psi \in C^1(\overline \Omega)$ fulfilling $\varphi> 0$ and $\psi>0$ in $\overline \Omega$ and $\int_{\Omega}\varphi^{\frac{3}{2}}\leq L$, we have
  \begin{align}\label{3.16}
    \int_{\Omega}\varphi^{\beta} \psi\leq \eta \int_{\Omega}\varphi^{p-1}\psi|\nabla
    \varphi|^2+\eta \int_{\Omega}\varphi\frac{|\nabla  \psi|^q}{{\psi}^{q-1}}+C\int_{\Omega}\varphi \psi.
\end{align}
\end{lemma}
\begin{proof}
Firstly, we verify in the case
\begin{equation}\label{3.17}
\beta \in \left[\frac{qp+q-2}{q-2},p+2\right).
\end{equation}
Let
$$\rho:=\frac{3}{3-2(\beta+p+1)}\quad and \quad \rho_*:=\frac{3}{2(\beta-p-1)}.$$
Due to \eqref{3.17}, it follows that
    \begin{align*}
        \rho\geq1\quad and\quad 1<2\rho<6.
    \end{align*}
Thanks to  $\int_{\Omega}\varphi^{\frac{3}{2}}\leq L$, the H\"{o}lder inequality, the Young inequality and Lemma \ref{lemma3.1},  we deduce that for any $\eta>0$, there exists
positive constant  $c_1=c_1(\eta,p,\beta,L)$  such that
\begin{align*}
    \int_{\Omega}\varphi^{\beta}\psi&=\int_{\Omega}(\varphi^{\frac{p+1}{2}}\psi^{\frac{1}{2}})^2\varphi^{\beta-p-1}\nonumber
    \\
    &\leq\|\varphi^{\frac{p+1}{2}}\psi^{\frac{1}{2}}\|^2_{L^{2\rho}(\Omega)}\cdot\|\varphi^{\beta-p-1}\|_{L^{\rho_*}(\Omega)}\nonumber
    \\
    &=\|\varphi^{\frac{p+1}{2}}\psi^{\frac{1}{2}}\|^2_{L^{2\rho}(\Omega)}
    \cdot\|\varphi\|_{L^{\frac3 2}(\Omega)}^{\beta-p-1}\nonumber
    \\
    &\leq L^{\frac {2\beta-2p-2}{3}}\|\varphi^{\frac{p+1}{2}}\psi^{\frac{1}{2}}\|^2_{L^{2\rho}(\Omega)}\nonumber
    \\
    &\leq \frac{\eta}{2}\int_{\Omega}\varphi^{p-1}\psi|\nabla
    \varphi|^2+\frac{\eta}{2}\int_{\Omega}\varphi^{p+1}\frac{|\nabla
    \psi|^2}{\psi}+c_1\int_{\Omega}\varphi\psi\nonumber
    \\
    &\leq \frac{\eta}{2}\int_{\Omega}\varphi^{p-1}\psi|\nabla
    \varphi|^2+\frac{\eta}{2}\int_{\Omega}\varphi\frac{|\nabla
    \psi|^q}{\psi^{q-1}}+\frac 12\int_{\Omega}\varphi^{\frac{qp+q-2}{q-2}}\psi+c_1\int_{\Omega}\varphi\psi\nonumber
     \\
    &\leq \frac{\eta}{2}\int_{\Omega}\varphi^{p-1}\psi|\nabla
    \varphi|^2+\frac{\eta}{2}\int_{\Omega}\varphi\frac{|\nabla
    \psi|^q}{\psi^{q-1}}+
    \frac{1}{2}\int_{\Omega}\varphi^{\beta}\psi+
    (c_1+1)\int_{\Omega}\varphi\psi
\end{align*}
 and thus
\begin{align}\label{3.18}
 \int_{\Omega}\varphi^{\beta}\psi\leq \eta\int_{\Omega}\varphi^{p-1}\psi|\nabla
\varphi|^2+\eta\int_{\Omega}\varphi\frac{|\nabla
\psi|^q}{\psi^{q-1}}+2 (c_1+1) \int_{\Omega}\varphi\psi.
\end{align}
Furthermore, it is observed that  for $\beta \in [1, \frac{qp+q-2}{q-2})$, we have $$\int_{\Omega}\varphi^{\beta}\psi\leq
    \int_{\Omega}\varphi^{\frac{qp+q-2}{q-2}}\psi+\int_{\Omega}\varphi\psi, $$
which along with \eqref{3.18}  completes the proof readily.
\end{proof}
Combining the improved integrability from Lemma \ref{lemma3.6} with the weighted gradient estimates in Lemma \ref{lemma2.7}, one can  establish $L^p$-estimates for $u_\varepsilon$ for any $p\geq2$ by means of the testing-based argument.
\begin{lemma}\label{lemma3.7}
Let $\Omega\subset\mathbb{R}^3$ and $p\geq2$. There exists $C=C(p)>0$ such that
    \begin{equation}\label{3.19}
        \begin{split}
          \int_\Omega u_\varepsilon^p\leq C,
        \end{split}
    \end{equation}
    and
    \begin{equation}\label{3.20}
    \int_0^\infty\int_\Omega u_{\varepsilon}^{p-1}v_{\varepsilon}|\nabla u_{\varepsilon}|^2\leq C,
    \end{equation}
    as well as
    \begin{equation}\label{3.21}
    \int_0^{\infty}\int_\Omega u_\varepsilon^{p+1}v_\varepsilon\leq C
    \end{equation} for all $\varepsilon>0$.
\end{lemma}
\begin{proof}
According to Lemma \ref{lemma3.5} and the Neumann heat semigroup estimates in Lemma \ref{lemma2.4}, we can obtain
\begin{equation}\label{3.22}
\|w_\varepsilon(\cdot,t)\|_{L^\rho(\Omega)}\leq c_1~~\mbox{for all}~~\rho\in(0,\infty),~~t\in(0,\infty).
\end{equation}
Multiplying the first equation of \eqref{2.5} by $u_\varepsilon^{p-1}$, integrating by parts, using Young's inequality and \eqref{3.22}, we obtain
    \begin{align}\label{3.23}
           &~\frac{1}{p}\frac{d}{dt}\int_\Omega u_\varepsilon^p\nonumber\\
           &=\int_\Omega u_\varepsilon^{p-1}\{\nabla\cdot(u_\varepsilon v_\varepsilon\nabla u_\varepsilon)-\nabla\cdot(u_\varepsilon^2v_\varepsilon\nabla v_\varepsilon)\}+\ell\int_\Omega u_\varepsilon^{p-1} v_\varepsilon w_\varepsilon\nonumber\\
           &=-(p-1)\int_\Omega u_\varepsilon^{p-1}v_\varepsilon|\nabla u_\varepsilon|^2+(p-1)\int_\Omega u_\varepsilon^pv_\varepsilon\nabla u_\varepsilon\cdot\nabla v_\varepsilon+\ell\int_\Omega u_\varepsilon^{p-1} v_\varepsilon w_\varepsilon\nonumber\\
           &\leq-\frac{p-1}{2}\int_\Omega u_\varepsilon^{p-1}v_\varepsilon|\nabla u_\varepsilon|^2+\frac{(p-1)}{2}\int_\Omega u_\varepsilon^{p+1}v_\varepsilon|\nabla v_\varepsilon|^2+\frac{\ell(p-1)}{p}\int_\Omega u_\varepsilon^p v_\varepsilon+\frac{\ell}{p}\int_\Omega v_\varepsilon w_\varepsilon^{p}\\
           &\leq -\frac{p-1}{2}\int_\Omega u_\varepsilon^{p-1}v_\varepsilon|\nabla u_\varepsilon|^2+c_2 \int_{\Omega}u_{\varepsilon}\frac{|\nabla v_{\varepsilon}|^q}{v_{\varepsilon}^{q-1}}+c_3\int_{\Omega}u_{\varepsilon}^{\frac{qp+q-2}{q-2}}v_{\varepsilon}^{\frac{3q-2}{q-2}}\nonumber\\
           &~~+\frac{\ell(p-1)}{p}\int_\Omega u_\varepsilon^p v_\varepsilon+\frac{\ell e^{2\kappa}\|v_0\|_{L^\infty(\Omega)}e^{-\kappa t}}{p}\int_\Omega w_\varepsilon^p\nonumber
    \end{align}
    for all $t\in(0,\infty)$. Noting that for any  fixed $p>1$, one can find $q>2p+2$  such that
\begin{align*}
    \frac{qp+q-2}{q-2}<p+2,
\end{align*}
which along with \eqref{3.16} yields
\begin{align}\label{3.24}
\begin{split}
&c_3\int_{\Omega}u_{\varepsilon}^{\frac{qp+q-2}{q-2}}v_{\varepsilon}^{\frac{3q-2}{q-2}}\\
   &\leq c_3\|v_0\|_{L^\infty(\Omega)}^{\frac{2q}{q-2}}\int_{\Omega}u_{\varepsilon}^{\frac{qp+q-2}{q-2}}v_{\varepsilon}
   \\
   &\leq\eta\int_{\Omega}u_{\varepsilon}^{p-1}v_{\varepsilon}|\nabla u_{\varepsilon}|^2+\eta\int_{\Omega}u_{\varepsilon}\frac{|\nabla
   v_{\varepsilon}|^q}{v_{\varepsilon}^{q-1}}+c_4(\eta,p,\|v_0\|_{L^\infty(\Omega)})\int_{\Omega}u_{\varepsilon}v_{\varepsilon}
\end{split}
\end{align}
for all $t\in (0,\infty)$.
Combining  \eqref{3.23} with \eqref{3.24} and applying \eqref{2.12}, we then have
    \begin{align}
             &\frac{1}{p}\frac{d}{dt}\int_\Omega u_\varepsilon^p+\frac{p-1}{4}\int_\Omega u_\varepsilon^{p-1}v_\varepsilon|\nabla u_\varepsilon|^2\nonumber\\
             &\leq (c_2+\eta) \int_{\Omega}u_{\varepsilon}\frac{|\nabla v_{\varepsilon}|^q}{v_{\varepsilon}^{q-1}}
             +c_4(\eta,p,\|v_0\|_{L^\infty(\Omega)})\int_{\Omega}u_{\varepsilon}v_{\varepsilon}+\ell\int_\Omega u_\varepsilon^p v_\varepsilon+\ell c_1\|v_0\|_{L^\infty(\Omega)}e^{-\kappa t}\nonumber\\
             &\leq c_5\int_\Omega\frac{|\nabla v_\varepsilon|^{\frac{pq}{p-1}}}{v_\varepsilon^{\frac{pq}{p-1}-1}}+(c_6+\ell)\int_\Omega u_\varepsilon^p v_\varepsilon+c_7(\ell,K)e^{-\kappa t}\\
             &\leq c_5\int_\Omega\frac{|\nabla v_\varepsilon|^{\frac{pq}{p-1}}}{v_\varepsilon^{\frac{pq}{p-1}-1}}+c_8(\ell,K)
             e^{-\kappa t}(\int_\Omega u_\varepsilon^p +1).\nonumber
    \end{align}
Therefore we infer that
$$y_{\varepsilon}(t):=\int_\Omega u_\varepsilon^p(\cdot,t)+1, \quad g_{\varepsilon}(t) :=\frac{p(p-1)}{4}\int_\Omega u_\varepsilon^{p-1}v_\varepsilon|\nabla u_\varepsilon|^2$$
 and $$ f_{\varepsilon}(t):= pc_8(\ell,K)e^{-\kappa t}\quad\hbox{as well as}\quad h_{\varepsilon}(t):=
 c_5p\int_\Omega\frac{|\nabla v_\varepsilon|^{\frac{pq}{p-1}}}{v_\varepsilon^{\frac{pq}{p-1}-1}}
$$ satisfies
$$  y'_{\varepsilon}(t)+ g_{\varepsilon}(t) \leq f_{\varepsilon}(t)y_{\varepsilon}(t) +h_{\varepsilon}(t) \quad\quad\quad \hbox{for all}\,\, t>0,
$$
which  upon a first integration shows that
\begin{equation}
\begin{split}y_{\varepsilon}(t)& \leq  y_{\varepsilon}(0)\cdot e^{\int^t_0 f_{\varepsilon}(s)ds}+\int^t_0
e^{\int^t_s f_{\varepsilon}(\sigma)d\sigma} h_{\varepsilon}(s)ds\\
&\leq y_{\varepsilon}(0)e^{ p\kappa^{-1}c_8}+e^{ p\kappa^{-1}c_8}\int^t_0
 h_{\varepsilon}(s)ds.
\end{split}
\end{equation}
Now  from Lemma \ref{lemma2.7} and Lemma \ref{lemma2.4}, we find $c_9>0$ such that
$$ \int^t_0
 h_{\varepsilon}(s)ds\leq c_9 \quad\quad\quad \hbox{for all}\,\, t>0 \,\hbox{ and}\,\, \varepsilon>0,
$$
which implies that
\begin{equation}\label{3.26}
    \int_\Omega u_\varepsilon^p(\cdot,t)+\int_0^{t}\int_\Omega u_\varepsilon^{p-1}v_\varepsilon|\nabla u_\varepsilon|^2\leq c_{10}
\end{equation}
for some $c_{10}>0$ and thereby
\eqref{3.21} results from Lemma \ref{lemma3.6}, Lemma \ref{lemma2.7} and \eqref{3.26}.
\end{proof}

The following lemma extends the two-dimensional result of \cite{WinklerEven} to that  in  three-dimensional cases, and thereby enabling us to establish the $L^{\infty}$-estimate for $u_\varepsilon$.
\begin{lemma}\label{lemma3.8}
Let $\Omega\subset\mathbb{R}^3$ be a smoothly bounded domain. Then there exist $\gamma>0$ and $C>0$ such that for any $p>4$ and $\eta\in(0,1]$, and for all $\varphi\in C(\overline{\Omega})$ and $\psi\in C(\overline{\Omega})$ fulfilling $\varphi>0$ and $\psi>0$ in $\overline{\Omega}$, we have
\begin{equation*}
\begin{split}
\int_\Omega \varphi^{p+1}\psi\leq\eta\int_\Omega\varphi^{p-1}\psi|\nabla\varphi|^2+\eta\bigg\{\int_\Omega\frac{|\nabla\psi|^8}{\psi^7}\bigg\}\cdot\bigg\{\int_\Omega\varphi^{\frac{p}{2}}\bigg\}^{\frac{2(p+1)} {p}}
+C\eta^{-\gamma}p^{2\gamma}\bigg\{\int_\Omega\varphi\psi\bigg\}\cdot\bigg\{\int_\Omega\varphi^{\frac{p}{2}}\bigg\}^{2}.
\end{split}
\end{equation*}
\end{lemma}
\begin{proof}
For fixed $q\in(\frac6 5, \frac4 3)$, the three-dimensional Gagliardo-Nirenberg inequality provides $c_1(q)>0$ such that
    \begin{equation*}
        \|\rho\|_{L^2(\Omega)}^2\leq c_1(q)\|\nabla\rho\|_{L^q(\Omega)}^{\frac{12q}{11q-6}}\|\rho\|_{L^{\frac2 3}(\Omega)}^{\frac{10q-12}{11q-6}}+c_1(q)\|\rho\|_{L^{\frac2 3}(\Omega)}^{2}~~\mbox{for all}~~\rho\in C^1(\overline{\Omega}).
    \end{equation*}
    Apply this inequality to $\rho=\varphi^{\frac{p+1}{2}}\psi^{\frac12}$, we obtain
    \begin{equation}\label{3.28}
    \begin{split}
        \|\varphi^{p+1}\psi\|_{L^1(\Omega)}&=\|\varphi^{\frac{p+1}{2}}\psi^{\frac1 2}\|_{L^2(\Omega)}^2\\
        &\leq c_1\|\nabla(\varphi^{\frac{p+1}{2}}\psi^{\frac1 2})\|_{L^q(\Omega)}^{\frac{12q}{11q-6}}\|\varphi^{\frac{p+1}{2}}\psi^{\frac1 2}\|_{L^{\frac2 3}(\Omega)}^{\frac{10q-12}{11q-6}}+c_1\|\varphi^{\frac{p+1}{2}}\psi^{\frac1 2}\|_{L^{\frac2 3}(\Omega)}^{2}
        \end{split}
    \end{equation}
   with $\varphi\in C^1(\overline{\Omega})$ and $\psi\in C^1(\overline{\Omega})$. By  Young's inequality, we can see that
    \begin{equation}\label{3.29}
    \begin{split}
       &c_1\|\nabla(\varphi^{\frac{p+1}{2}}\psi^{\frac1 2})\|_{L^q(\Omega)}^{\frac{12q}{11q-6}}\|\varphi^{\frac{p+1}{2}}\psi^{\frac1 2}\|_{L^{\frac2 3}(\Omega)}^{\frac{10q-12}{11q-6}}\\
       &=\bigg\{\delta\|\nabla(\varphi^{\frac{p+1}{2}}\psi^{\frac1 2})\|_{L^q(\Omega)}^{2}\bigg\}^{\frac{6q}{11q-6}}\cdot c_1\delta^{-\frac{6q}{11q-6}}\|\varphi^{\frac{p+1}{2}}\psi^{\frac1 2}\|_{L^{\frac2 3}(\Omega)}^{\frac{10q-12}{11q-6}}\\
       &\leq\delta\|\nabla(\varphi^{\frac{p+1}{2}}\psi^{\frac1 2})\|_{L^q(\Omega)}^{2}+c_1^{\frac{11q-6}{5q-6}}\delta^{-\frac{6q}{5q-6}}\|\varphi^{\frac{p+1}{2}}\psi^{\frac1 2}\|_{L^{\frac2 3}(\Omega)}^{2}\\
       &\leq\delta\bigg\{\frac{p+1}{2}\|\varphi^{\frac{p-1}{2}}\psi^{\frac1 2}\nabla\varphi\|_{L^q(\Omega)}+\frac{1}{2}\|\varphi^{\frac{p+1}{2}}\psi^{-\frac1 2}\nabla\psi\|_{L^q(\Omega)}\bigg\}^2
       +c_1^{\frac{11q-6}{5q-6}}\delta^{-\frac{6q}{5q-6}}\|\varphi^{\frac{p+1}{2}}\psi^{\frac1 2}\|_{L^{\frac2 3}(\Omega)}^{2}\\
       &\leq\frac{(p+1)^2 \delta}{2}\|\varphi^{\frac{p-1}{2}}\psi^{\frac1 2}\nabla\varphi\|_{L^q(\Omega)}^2+\frac{\delta}{2}\|\varphi^{\frac{p+1}{2}}\psi^{-\frac1 2}\nabla\psi\|_{L^q(\Omega)}^2
       +c_1^{\frac{11q-6}{5q-6}}\delta^{-\frac{6q}{5q-6}}\|\varphi^{\frac{p+1}{2}}\psi^{\frac1 2}\|_{L^{\frac2 3}(\Omega)}^{2}.\\
        \end{split}
    \end{equation}
   Moreover, due to $q<2$,  the application of the H\"{o}lder inequality yields
    \begin{equation}
        \begin{split}
            \frac{(p+1)^2 \delta}{2}\|\varphi^{\frac{p-1}{2}}\psi^{\frac1 2}\nabla\varphi\|_{L^q(\Omega)}^2&\leq\frac{(p+1)^2 \delta}{2}\cdot|\Omega|^{\frac{2-q}{q}}\|\varphi^{\frac{p-1}{2}}\psi^{\frac1 2}\nabla\varphi\|_{L^2(\Omega)}^2.\\
        \end{split}
    \end{equation}
    On the other hand, using the H\"{o}lder inequality and the Young inequality once more, one can find $c_2>0$ such that
    \begin{align}
            \frac{\delta}{2}\|\varphi^{\frac{p+1}{2}}\psi^{-\frac1 2}\nabla\psi\|_{L^q(\Omega)}^2&=\frac{\delta}{2}\cdot\bigg\{\int_\Omega\varphi^{\frac{(p+1)q}{2}}\psi^{-\frac q 2}|\nabla\psi|^q\bigg\}^{\frac2 q}\nonumber\\
            &=\frac{\delta}{2}\cdot\bigg\{\int_\Omega\bigg(\frac{|\nabla\psi|^8}{\psi^7}\bigg)^{\frac{q}{8}}\varphi^{\frac{(p+1)q}{2}}\psi^{\frac {3q} 8}\bigg\}^{\frac2 q}\nonumber\\
            &\leq\frac{\delta}{2}\cdot\bigg\{\int_\Omega\frac{|\nabla\psi|^8}{\psi^7}\bigg\}^{\frac{1}{4}}\bigg\{\int_\Omega\varphi^{\frac{4(p+1)q}{8-q}}\psi^{\frac {3q} {8-q}}\bigg\}^{\frac{8-q} {4q}}\nonumber\\
            &=\frac{\delta}{2}\cdot\bigg\{\int_\Omega\frac{|\nabla\psi|^8}{\psi^7}\bigg\}^{\frac{1}{4}}\bigg\{\int_\Omega\big(\varphi^{p+1}\psi\big)^{\frac {3q} {8-q}}\varphi^{\frac{(p+1)q}{8-q}}\bigg\}^{\frac{8-q} {4q}}\\
            &\leq\frac{\delta}{2}\cdot
            \bigg\{\int_\Omega\frac{|\nabla\psi|^8}{\psi^7}\bigg\}^{\frac{1}{4}}
            \bigg\{\int_\Omega\varphi^{p+1}\psi\bigg\}^{\frac{3} {4}}\bigg\{\int_\Omega\varphi^{\frac{(p+1)q}{8-4q}}\bigg\}^{\frac{2-q} {q}}\nonumber\\
             &\leq\frac{c_2\delta}{2}\cdot
             \bigg\{\int_\Omega\frac{|\nabla\psi|^8}{\psi^7}\bigg\}^{\frac{1}{4}}\bigg\{\int_\Omega\varphi^{p+1}\psi\bigg\}^{\frac{3} {4}}\bigg\{\int_\Omega\varphi^{\frac{p}{2}}\bigg\}^{\frac{p+1} {2p}}\nonumber\\
             &\leq\frac{1}{2}\int_\Omega\varphi^{p+1}\psi+\frac{c_2^4\delta^4}{2}\bigg\{\int_\Omega\frac{|\nabla\psi|^8}{\psi^7}\bigg\}\cdot\bigg\{\int_\Omega\varphi^{\frac{p}{2}}\bigg\}^{\frac{2(p+1)} {p}},\nonumber
        \end{align}
where the fact that $p>4$ and $\frac6 5 <q<\frac4 3$ implies $\frac{(p+1)q}{8-4q}\leq\frac{p}{2}$ and thus allows us to apply the H\"{o}lder inequality to get
\begin{align*}
\bigg\{\int_\Omega\varphi^{\frac{(p+1)q}{8-4q}}\bigg\}^{\frac{2-q} {q}}
&\leq|\Omega|^{\frac{4p-3pq-q}{2pq}}
\bigg\{\int_\Omega\varphi^{\frac{p}{2}}\bigg\}^{\frac{p+1} {2p}}\\
& \leq c_2\bigg\{\int_\Omega\varphi^{\frac{p}{2}}\bigg\}^{\frac{p+1} {2p}}
\end{align*}
with $c_2:=\max\{1,|\Omega|^{\frac{2}{q}}\}$ due to $0\leq\frac{4p-3pq-q}{2pq}\leq\frac{2}{q}$.
    In addition, we have
    \begin{equation}\label{3.32}
        \|\varphi^{\frac{p+1}{2}}\psi^{\frac{1}{2}}\|_{L^{\frac{2}{3}}(\Omega)}^2=\bigg\{\int_\Omega\varphi^{\frac{p+1}{3}}\psi^{\frac{1}{3}}\bigg\}^{3}=\bigg\{\int_\Omega\big(\varphi\psi\big)^{\frac{1}{3}}\varphi^{\frac{p}{3}}\bigg\}^{3}\leq\bigg\{\int_\Omega\varphi\psi\bigg\}\cdot\bigg\{\int_\Omega\varphi^{\frac{p}{2}}\bigg\}^{2}.
    \end{equation}
    Therefore  for given $\eta\in(0,1]$ and $\delta:=\min\big\{\frac{\eta}{(p+1)^2|\Omega|^{\frac{2-q}{q}}},\frac{\eta^{\frac{1}{4}}}{c_2}\big\}$, we combine \eqref{3.28}--\eqref{3.32} to arrive at
    \begin{align}\label{3.33}
            &\int_\Omega\varphi^{p+1}\psi\nonumber\\
            ~~&\leq(p+1)^2 \delta|\Omega|^{\frac{2-q}{q}}\|\varphi^{\frac{p-1}{2}}\psi^{\frac1 2}\nabla\varphi\|_{L^2(\Omega)}^2+c_2^4\delta^4\bigg\{\int_\Omega\frac{|\nabla\psi|^8}{\psi^7}\bigg\}\cdot\bigg\{\int_\Omega\varphi^{\frac{p}{2}}\bigg\}^{\frac{2(p+1)} {p}}\nonumber\\
            &~~+c_3
            \bigg\{\int_\Omega\varphi\psi\bigg\}
            \cdot\bigg\{\int_\Omega\varphi^{\frac{p}{2}}\bigg\}^{2}\\
            &\leq\eta\int_\Omega\varphi^{p-1}\psi|\nabla\varphi|^2
            +\eta\bigg\{\int_\Omega\frac{|\nabla\psi|^8}{\psi^7}\bigg\}
            \cdot\bigg\{\int_\Omega\varphi^{\frac{p}{2}}\bigg\}^{\frac{2(p+1)} {p}}
            +c_3
            \bigg\{\int_\Omega\varphi\psi\bigg\}
            \cdot\bigg\{\int_\Omega\varphi^{\frac{p}{2}}\bigg\}^{2}\nonumber\\
            &\leq\eta\int_\Omega\varphi^{p-1}\psi|\nabla\varphi|^2+\eta\bigg\{\int_\Omega\frac{|\nabla\psi|^8}{\psi^7}\bigg\}\cdot\bigg\{\int_\Omega\varphi^{\frac{p}{2}}\bigg\}^{\frac{2(p+1)} {p}}
            +c_4\eta^{-\gamma}p^{2\gamma}\bigg\{\int_\Omega\varphi\psi\bigg\}\cdot\bigg\{\int_\Omega\varphi^{\frac{p}{2}}\bigg\}^{2}\nonumber
    \end{align}
    where the inequalities $p>1$ and $\eta\leq1$ warrant that
    \begin{align}\label{3.34}
            c_3&:=2c_1^{\frac{11q-6}{5q-6}}\delta^{-\frac{6q}{5q-6}}+2c_1\nonumber\\
            &=2c_1^{\frac{11q-6}{5q-6}}\cdot
            \max\bigg\{\bigg(\frac{(p+1)^2|\Omega|^{\frac{2-q}{q}}}{\eta}\bigg)^{\frac{6q}{5q-6}},
            \bigg(\frac{c_2}{\eta^{\frac{1}{4}}}\bigg)^{\frac{6q}{5q-6}}\bigg\}+2c_1\\
            &\leq 2c_1^{\frac{11q-6}{5q-6}}\cdot\max\bigg\{(4|\Omega|^{\frac{2-q}{q}})^{\frac{6q}{5q-6}},
            c_2^{\frac{6q}{5q-6}}\bigg\}
            \eta^{-\frac{6q}{5q-6}}p^{\frac{12q}{5q-6}}
            +2c_1\eta^{-\frac{6q}{5q-6}}p^{\frac{12q}{5q-6}}\nonumber\\
            & \leq c_4\eta^{-\gamma}p^{2\gamma}\nonumber
        \end{align}
    with $c_4:=2c_1^{\frac{11q-6}{5q-6}}\cdot\max\{(4|\Omega|^{\frac{2-q}{q}})^{\frac{6q}{5q-6}},c_2^{\frac{6q}{5q-6}}\}+2c_1$ and $\gamma=\frac{6q}{5q-6}$. The claim hence results from \eqref{3.33} and \eqref{3.34}.
\end{proof}
\begin{lemma}\label{lemma3.9}(\cite{WinklerEven})
    Let $a\geq1,b\geq1,q\geq0$ and $\{N_k\}_{k\in\{1,2,3,...\}}\subset[1,\infty)$ be such that
    \begin{equation*}
        N_{k}\leq a^k N_{k-1}^{2+q\cdot 2^{-k}}+b^{2^k}~~\mbox{for all}~~k\geq1.
    \end{equation*}
    Then
    \begin{equation}
        \liminf\limits_{k\to\infty}N_k^{\frac{1}{2^k}}\leq(2\sqrt{2}a^3b^{1+\frac{q}{2}}N_0)^{e^{\frac{q}{2}}}.
    \end{equation}
\end{lemma}
\begin{proof}[Proof of Theorem \ref{theorem 1.1}] The global existence has fully been covered by Lemmas \ref{lemma2.11}--\ref{lemma2.15}. According to Lemma \ref{lemma3.7}, there exists $c_1(p)>0$ such that $\|u(\cdot,t)\|_{L^p(\Omega)}\leq c_1(p)$ for all $p\geq1$. Subsequently, by applying smoothing estimates for the Neumann heat semigroup, we obtain $\|w(\cdot,t)\|_{L^{\infty}(\Omega)}\leq c_2$ and $\|\nabla v(\cdot,t)\|_{L^\infty(\Omega)}\leq c_3$ for all $t>0$. For integers $k\geq1$, we set $p_k:=3\cdot2^k$, and let
\begin{equation}\label{3.36}
    N_k=1+\sup\limits_{\varepsilon\in(0,1)}\sup\limits_{t>0}\int_\Omega u_\varepsilon^{p_k}(\cdot,t).
\end{equation}
By the first equation of system \eqref{2.5}  and Young's inequality, one can see that
\begin{align*}
        \frac{d}{dt}\int_\Omega u_\varepsilon^{p_k}&=-p_k(p_k-1)\int_\Omega u_\varepsilon^{p_k-1}v_\varepsilon|\nabla u_\varepsilon|^2+p_k(p_k-1)\int_\Omega u_\varepsilon^{p_k}v_\varepsilon\nabla u_\varepsilon\cdot\nabla v_\varepsilon+p_k\int_\Omega v_\varepsilon w_\varepsilon u_\varepsilon^{p_k-1}\\
        &\leq-\frac{p_k(p_k-1)}{2}\int_\Omega u_\varepsilon^{p_k-1}v_\varepsilon|\nabla u_\varepsilon|^2+\frac{p_k(p_k-1)}{2}\int_\Omega u_\varepsilon^{p_k+1}v_\varepsilon|\nabla v_\varepsilon|^2\\
        &~~+\frac{c_2p_k(p_k-1)}{p_k+1}\int_\Omega v_\varepsilon u_\varepsilon^{p_k+1}+\frac{2c_2 p_k}{p_k+1}\int_\Omega v_\varepsilon\\
        &\leq-\frac{p_k(p_k-1)}{2}\int_\Omega u_\varepsilon^{p_k-1}v_\varepsilon|\nabla u_\varepsilon|^2+(c_3^2+c_2)p_k(p_k-1)\int_\Omega u_\varepsilon^{p_k+1}v_\varepsilon
        +2c_2p_k\int_\Omega v_\varepsilon.
\end{align*}
Since $p_k(p_k-1)\leq p_k^2,p_k\leq p_k^2$ and $\frac{p_k(p_k-1)}{2}\geq\frac{p_k}{4}$, we infer that
\begin{equation}\label{3.37}
    \begin{split}
        \frac{d}{dt}\int_\Omega u_\varepsilon^{p_k}+\frac{p_k^2}{4}\int_\Omega u_\varepsilon^{p_k-1}v_\varepsilon|\nabla u_\varepsilon|^2\leq c_4p_k^2\int_\Omega u_\varepsilon^{p_k+1}v_\varepsilon+c_4p_k^2\int_\Omega v_\varepsilon
    \end{split}
\end{equation}
with $c_4=\max\{c_3^2+c_2,2c_2\}$. Applying Lemma \ref{lemma3.8} to $p:=6$, we can conclude that
\begin{equation*}
    \begin{split}
        \int_\Omega u_\varepsilon^{p_k+1}v_\varepsilon&\leq\frac{1}{4c_4}\int_\Omega u_\varepsilon^{p_k-1}v_\varepsilon|\nabla u_\varepsilon|^2+\frac{1}{4c_4}\bigg\{\int_\Omega\frac{|\nabla v_\varepsilon|^8}{v_\varepsilon^7}\bigg\}\cdot\bigg\{\int_\Omega u_\varepsilon^{\frac{p_k}{2}}\bigg\}^{\frac{2(p_k+1)} {p_k}}\\
        &~~+c_5(4c_4)^{\gamma}p_k^{2\gamma}\bigg\{\int_\Omega u_\varepsilon v_\varepsilon\bigg\}\cdot\bigg\{\int_\Omega u_\varepsilon^{\frac{p_k}{2}}\bigg\}^{2}.
    \end{split}
\end{equation*}
According to the definition of $N_k$ in \eqref{3.36}, we  have
\begin{equation*}
    \begin{split}
        \int_\Omega u_\varepsilon^{\frac{p_k}{2}}=\int_\Omega u_\varepsilon^{p_{k-1}}\leq N_{k-1}.
    \end{split}
\end{equation*}
Hence we obtain that
\begin{equation*}
    \begin{split}
        c_4p_k^2\int_\Omega u_\varepsilon^{p_k+1}v_\varepsilon&\leq\frac{p_k^2}{4}\int_\Omega u_\varepsilon^{p_k-1}v_\varepsilon|\nabla u_\varepsilon|^2+\frac{p_k^2}{4}N_{k-1}^{\frac{2(p_k+1)} {p_k}}\int_\Omega\frac{|\nabla v_\varepsilon|^8}{v_\varepsilon^7}\\
        &~~+c_5\cdot4^\gamma c_4^{\gamma+1}p_k^{2\gamma+2} N_{k-1}^{2}\int_\Omega u_\varepsilon v_\varepsilon.
    \end{split}
\end{equation*}
Since $p_k^2\leq p_k^{2\gamma+2}$ and $1\leq N_{k-1}^2\leq N_{k-1}^{\frac{2(p_k+1)}{p_k}}$ for all $k\geq1$, it follows from \eqref{3.37} that
\begin{equation}\label{3.38}
\begin{split}
     \frac{d}{dt}\int_\Omega u_\varepsilon^{p_k}&\leq \frac{p_k^2}{4}N_{k-1}^{\frac{2(p_k+1)} {p_k}}\int_\Omega\frac{|\nabla v_\varepsilon|^8}{v_\varepsilon^7}+c_5\cdot4^\gamma c_4^{\gamma+1}p_k^{2\gamma+2} N_{k-1}^{2}\int_\Omega u_\varepsilon v_\varepsilon+c_4p_k^2\int_\Omega v_\varepsilon\\
     &\leq c_6 p_k^{2\gamma+2} N_{k-1}^{\frac{2(p_k+1)}{p_k}}\bigg\{\int_\Omega\frac{|\nabla v_\varepsilon|^8}{v_\varepsilon^7}+\int_\Omega u_\varepsilon v_\varepsilon+\int_\Omega v_\varepsilon\bigg\}
     \end{split}
\end{equation}
with $c_6:=\max\{\frac{1}{4},c_5\cdot4^{\gamma}c_4^{\gamma+1},c_4\}$.
According to \eqref{2.15}, \eqref{2.8} and \eqref{2.12}, we find $c_7>0$ such that
\begin{equation}\label{3.39}
    \int_0^{\infty}\int_\Omega\bigg\{\frac{|\nabla v_\varepsilon|^8}{v_\varepsilon^7}+\int_\Omega u_\varepsilon v_\varepsilon+\int_\Omega v_\varepsilon\bigg\}\leq c_7.
\end{equation}
According to \eqref{3.36},  \eqref{3.38} and \eqref{3.39}, we arrive at
\begin{equation*}
    \begin{split}
        N_k&=1+\sup\limits_{\varepsilon\in(0,1)}\sup\limits_{t>0}\int_\Omega u_\varepsilon^{p_k}\\
    &\leq1+\int_\Omega(u_0+\varepsilon)^{p_k}+c_6c_7p_k^{2\gamma+2} N_{k-1}^{\frac{2(p_k+1)}{p_k}}\\
    &\leq1+|\Omega|\cdot\|u_0+1\|_{L^\infty(\Omega)}^{p_k}+c_6c_7p_k^{2\gamma+2} N_{k-1}^{\frac{2(p_k+1)}{p_k}}\\
    &\leq b^{2^k}+a^k N_{k-1}^{2+\frac2 3\cdot2^{-k}}~~\mbox{for all}~~k\geq1
    \end{split}
\end{equation*}
with $b:=1+|\Omega|\cdot\|u_0+1\|_{L^\infty(\Omega)}^3$ and $a:=\max\{6c_6c_7,1\}^{2\gamma+2}$. An application of Lemma \ref{lemma3.9} thus asserts that
\begin{equation*}
        \liminf\limits_{k\to\infty}N_k^{\frac{1}{2^k}}\leq(2\sqrt{2}a^3b^{\frac4 3}N_0)^{e^{\frac{1}{3}}},
    \end{equation*}
    which yields
    \begin{equation}\label{3.40}
    \|u_\varepsilon(\cdot,t)\|_{L^{\infty}(\Omega)}\leq C~~\mbox{for all}~~ t>0.
    \end{equation} The proof of Theorem \ref{theorem 1.1} is hence completed.
\end{proof}
\section{Large time behavior }
The aim of this section is to establish both the large-time convergence properties (Theorem \ref{theorem1.2}) and the occurrence of non-trivial pattern formation (Theorem \ref{Th1.3}). First, we establish the results on  large time decay of
 $u_{\varepsilon t}$ in generalized function spaces, which serves as a fundamental prerequisite for subsequent convergence analysis.
\begin{lemma}\label{lemma4.1}
Let $\Omega \subset \mathbb R^3$  be a smoothly bounded domain and $K>0$ with the property that \eqref{1.6} holds. Then there exist  $\sigma>0$ and $C=C(K)>0$ such that for all $\varepsilon\in(0,1)$ we have
\begin{align}\label{4.1}
   \int_0^{\infty}\|u_{\varepsilon t}(\cdot,t)\|_{(W^{1,\infty}(\Omega))^*}dt\leq C\cdot \left\{\int_\Omega v_{0}\right\}^{\sigma}.
\end{align}
\end{lemma}
\begin{proof}
    By the first equation in \eqref{2.5}, one can see that for all $t>0$ and any $\psi\in W^{1,\infty}(\Omega) $ such that $\|\psi\|_{W^{1,\infty}(\Omega)}\equiv \max\{\|\psi\|_{L^{\infty}(\Omega)},\|\nabla \psi\|_{L^{\infty}(\Omega)}\}\leq 1$,
    \begin{align*}
   \left |\int_{\Omega}u_{\varepsilon t}\psi\right|&=\left|-\int_{\Omega}u_{\varepsilon}v_{\varepsilon}\nabla u_{\varepsilon}\cdot\nabla\psi + \int_{\Omega}u_{\varepsilon}^{2}v_{\varepsilon}\nabla v_{\varepsilon}\cdot\nabla\psi+\ell\int_{\Omega}v_{\varepsilon}w_{\varepsilon}\psi\right|\\
    &\leq \int_{\Omega}u_{\varepsilon}v_{\varepsilon}|\nabla u_{\varepsilon}| + \int_{\Omega}u_{\varepsilon}^{2}v_{\varepsilon}|\nabla v_{\varepsilon}|+\ell\int_{\Omega}v_{\varepsilon}w_{\varepsilon},
    \end{align*}
 so that
 \begin{align}\label{4.2}
     \|u_{\varepsilon t}(\cdot,t)\|_{(W^{1,\infty}(\Omega))^*}\leq \int_{\Omega}u_{\varepsilon}v_{\varepsilon}|\nabla u_{\varepsilon}| +\int_{\Omega}u_{\varepsilon}^{2}v_{\varepsilon}|\nabla v_{\varepsilon}|+\ell\int_{\Omega}v_{\varepsilon}w_{\varepsilon}.
 \end{align}
 Thanks to the H\"{o}lder inequality and \eqref{2.10},  we infer that
  \begin{equation}\label{4.3}
     \int_0^{\infty}\int_{\Omega}u_{\varepsilon}v_{\varepsilon}|\nabla u_{\varepsilon}|
     \leq\left\{\int_0^{\infty}\int_{\Omega}u_{\varepsilon}v_{\varepsilon}|\nabla u_{\varepsilon}|^2\right\}^{\frac{1}{2}}\cdot \left\{\int_0^{\infty}\int_{\Omega}u_{\varepsilon}^{2}v_{\varepsilon} \right\}^{\frac{1}{4}}\cdot \left\{\int_0^{\infty}\int_{\Omega}v_{\varepsilon} \right\}^{\frac{1}{4}}\\
 \end{equation}
 and
  \begin{equation}\label{4.4}
     \int_0^{\infty}\int_{\Omega}u_{\varepsilon}^{2}v_{\varepsilon}|\nabla v_{\varepsilon}|
     \leq\left\{\int_0^{\infty}\int_{\Omega}v_{\varepsilon}|\nabla v_{\varepsilon}|^2\right\}^{\frac{1}{2}}\cdot \left\{\int_0^{\infty}\int_{\Omega}u_{\varepsilon}^{8}v_{\varepsilon} \right\}^{\frac{1}{4}}\cdot\left\{\int_0^{\infty}\int_{\Omega}v_{\varepsilon} \right\}^{\frac{1}{4}}\\
 \end{equation}
 as well as
 \begin{align}\label{4.5}
    \int_0^{\infty}\int_{\Omega}v_{\varepsilon}w_{\varepsilon}\leq\int_\Omega v_0.
 \end{align}
 Integrating the second equation in \eqref{2.5} and using Lemma \ref{lemma2.2}, we obtain
 \begin{equation*}
 \frac{d}{dt}\int_\Omega v_\varepsilon(\cdot,t)=-\int_{\Omega}v_\varepsilon w_\varepsilon\leq-\kappa\int_\Omega v_\varepsilon~~\mbox{for all}~~t>2,
 \end{equation*}
 which implies
 \begin{equation}\label{4.6*}
 \int_\Omega v_\varepsilon(\cdot,t)\leq e^{2\kappa}\|v_0\|_{L^1(\Omega)}e^{-\kappa t}~~\mbox{for all}~~t>0.
 \end{equation}
 Moreover, multiplying the second equation in \eqref{2.5} by $v_{\varepsilon}^2$, we have
    \begin{align}\label{4.6}
    \frac{d}{dt}\int_{\Omega}v_{\varepsilon}^3+6\int_{\Omega}v_{\varepsilon}|\nabla
    v_{\varepsilon}|^2=-3\int_{\Omega}v_{\varepsilon}^3w_{\varepsilon}\leq 0.
\end{align}
Integrating \eqref{4.6} and using \eqref{1.6}, we then obtain \begin{equation}\label{4.7}\int_0^{\infty}\int_\Omega v_\varepsilon|\nabla v_\varepsilon|^2\leq c.
\end{equation}
 It follows from \eqref{4.2}--\eqref{4.7} that there exists $C(K)>0$ such that
 \begin{align*}
      \int_0^{\infty}\|u_{\varepsilon t}(\cdot,t)\|_{(W^{1,\infty}(\Omega))^*}dt\leq C(K)\cdot\left\{\int_\Omega v_{0}\right\}^{\sigma}
 \end{align*}
 due to \eqref{1.6}, \eqref{3.20}, and \eqref{3.21}, and thereby we arrive at \eqref{4.1} with $\sigma:=\max\{\frac{1}{4},1\}$.
\end{proof}
For the subsequent reasoning, we formulate the results of  Lemma  \ref{lemma4.1}, which  exclusively involves the zero-order expression of $u$.
\begin{lemma}\label{lemma4.2}
Let $K>0$ with the property that \eqref{1.6} holds. Given
  $\sigma>0$ and $C=C(K)>0$ as in Lemma \ref{lemma4.1}, then for any nondecreasing $(t_{k})_{k\in \mathbb{N}}\subset [0,\infty)$, we have
\begin{align}\label{4.8}
   \sum_{k=1}^{\infty}\|u(\cdot,t_{k+1})-u(\cdot,t_{k})\|_{(W^{1,\infty}(\Omega))^*}dt\leq C\cdot \left\{\int_\Omega v_{0}\right\}^{\sigma},
\end{align}
where  we set $u(\cdot,0):=u_0$.
\end{lemma}
\begin{proof}
     Fixing any such nondecreasing sequence $(t_k)_{k \in \mathbb{N}}$, we infer from Lemma \ref{lemma4.1} that
     \begin{align*}
         \sum_{k\in \mathbb{N}} \| u_{\varepsilon}(\cdot, t_{k+1}) - u_{\varepsilon}(\cdot, t_k) \|_{(W^{1,\infty}(\Omega))^*}
    &= \sum_{k\in \mathbb{N}} \left\| \int_{t_k}^{t_{k+1}} u_{\varepsilon t}(\cdot, t) \, dt \right\|_{(W^{1,\infty}(\Omega))^*}\\
    &\leq \sum_{k\in \mathbb{N}} \int_{t_k}^{t_{k+1}} \| u_{\varepsilon}(\cdot, t) \|_{(W^{1,\infty}(\Omega))^*} \, dt\\
    &\leq C(K) \cdot \left\{\int_\Omega v_{0}\right\}^{\sigma},
     \end{align*}
     because $(t_k, t_{k+1})\cap(t_l, t_{l+1})=\emptyset$ for $k \in \mathbb{N}$ and $l \in \mathbb{N}$ with $k \neq l$. This along with \eqref{2.41} implies \eqref{4.8} immediately.
\end{proof}
The quantitative dependence on $v_0$ not only provides large-time stabilization of individual trajectories in their first component but also quantifies the proximity between the limiting profile and initial data.
\begin{lemma}\label{lemma4.3}
Let $K>0$ with the property that \eqref{1.6} holds.  Then  the function $u$ obtained in Lemma \ref{lemma4.2} exhibits the convergence
\begin{align}\label{4.9}
    u(\cdot,t)\rightarrow u_{\infty}\qquad in\;(W^{1,\infty}(\Omega))^*\qquad as\; t\to \infty
\end{align}
with some $u_{\infty}\in(W^{1,\infty}(\Omega))^*$ which satisfies
\begin{align}\label{4.10}
    \|u_{\infty}-u_0\|_{(W^{1,\infty}(\Omega))^*}\leq C(K)\cdot \left\{\int_\Omega v_{0}\right\}^{\sigma}
\end{align}
with $C(K)>0$  as given in Lemma \ref{lemma4.1}.
\end{lemma}
\begin{proof}
    Lemma \ref{lemma4.2} implies that $\{u(\cdot,t_{k})\}_{k\in \mathbb{N}}$ forms a Cauchy sequence in $(W^{1,\infty}(\Omega))^*$, which establishes \eqref{4.9} with some $u_{\infty}\in(W^{1,\infty}(\Omega))^*$. Now select the sequence $(t_{k})_{k\in \mathbb{N}}$ with $t_1:=0$ and $t_k:=t$ for $k\geq 2$ in \eqref{4.8}, yielding
    \begin{align}\label{4.11}
        \|u(\cdot,t)-u_0\|_{(W^{1,\infty}(\Omega))^*}\leq C(K)\cdot \left\{\int_\Omega v_{0}\right\}^{\sigma}\qquad for \;all\; t>0,
\end{align}
and thereby derives \eqref{4.10} due to \eqref{4.9}.
\end{proof}

The quantitative form of the right-hand side in \eqref{4.10} and \eqref{4.11} allow us to derive the following stability property  of function pairs $(u_0,0)$.
\begin{lemma}\label{lemma4.4}
Let $K>0$ with the property that \eqref{1.6} holds.  Then for each $\eta>0$, there exists $\delta_1=\delta_1(K,\eta)>0$ whenever $u_0, v_0$ and $w_0$ fulfill
\eqref{1.2}, as well as
\begin{align}\label{4.12}
    \int_{\Omega}v_0\leq \delta_1,
\end{align}
the solution $(u,v,w)$ of \eqref{1.1} obtained in Theorem \ref{theorem 1.1} satisfies
\begin{align}\label{4.13}
    \|u(\cdot,t)-u_0\|_{(W^{1,\infty}(\Omega))^*}\leq \eta\qquad for \;all\; t>0.
\end{align}
Moreover, the corresponding limit function  from Lemma \ref{lemma4.3} admits
\begin{align}\label{4.14}
   \|u_{\infty}-u_0\|_{(W^{1,\infty}(\Omega))^*}\leq \eta.
\end{align}
\end{lemma}
\begin{proof}
   From \eqref{4.10} and \eqref{4.11}, it follows that for any $\eta>0$,  there exists $\delta=\delta(K,\eta)>0$ such that \eqref{4.12} warrants \eqref{4.13} and \eqref{4.14}.
\end{proof}

\begin{proof}[Proof of Theorem \ref{theorem1.2}] The claimed result has precisely been asserted by Lemma \ref{lemma4.4}.
\end{proof}
\begin{lemma}\label{lemma4.5*}
Let $u_{\infty}$ be as defined in Lemma \ref{lemma4.3}. Then
\begin{equation}\label{4.16*}
u(\cdot,t)\to u_{\infty}~~\mbox{in}~~L^{\infty}(\Omega)~~\mbox{as}~~t\to\infty.
\end{equation}
\end{lemma}
\begin{proof}
With $(u_{\varepsilon}, v_{\varepsilon},w_\varepsilon)$ and $(\varepsilon_{j})_{j \in \mathbb{N}}$ taken from Lemma \ref{lemma2.1}, let
$$L_{\varepsilon} := \int_{0}^{\infty} \|v_{\varepsilon}(\cdot, t)\|_{L^{\infty}(\Omega)} dt, \quad \varepsilon \in (\varepsilon_{j})_{j \in \mathbb{N}},$$
$$\tau := \phi_{\varepsilon}(t) := \frac{1}{L_{\varepsilon}} \int_{0}^{t} \|v_{\varepsilon}(\cdot, s)\|_{L^{\infty}(\Omega)} ds, \quad t \geq 0$$
and
\begin{equation}\label{4.17*}
m_{\varepsilon}(x, \tau) := u_{\varepsilon}(x, \phi_{\varepsilon}^{-1}(\tau)), \quad x \in \overline{\Omega}, \quad \tau \in [0, 1).
\end{equation}
Then we have
\begin{equation}\label{4.1*}
\left\{
\begin{array}{ll}
m_{\varepsilon \tau} = \nabla \cdot (a_{\varepsilon}(x, \tau) m_{\varepsilon}\nabla m_{\varepsilon}) - \nabla \cdot (b_{\varepsilon}(x, \tau) m_{\varepsilon}^2) + \ell a_{\varepsilon}(x, \tau) w_{\varepsilon}, & x \in \Omega, \quad \tau \in (0, 1), \\
\nabla m_{\varepsilon} \cdot \nu = 0, & x \in \partial \Omega, \quad \tau \in (0, 1),\\
m_{\varepsilon}(x, 0)=m_{0}(x)+\varepsilon, & x \in \Omega\\
\end{array}
\right.
\end{equation}
with
$$a_{\varepsilon}(x, \tau) := L_{\varepsilon} \cdot \frac{v_{\varepsilon}(x, t)}{\|v_{\varepsilon}(\cdot, t)\|_{L^{\infty}(\Omega)}} \quad \text{and} \quad b_{\varepsilon}(x, \tau) := L_{\varepsilon} \cdot \frac{v_{\varepsilon}(x, t) \nabla v_{\varepsilon}(x, t)}{\|v_{\varepsilon}(\cdot, t)\|_{L^{\infty}(\Omega)}}.$$
According to Harnack-type inequality (see \cite{HarnackJDE}), 
there exists  constant $\lambda_*>0$ such that
\begin{equation}\label{4.19*}
v_{\varepsilon}(x,t)\geq \lambda_* \|v_{\varepsilon}(\cdot,t)\|_{L^{\infty}(\Omega)}\quad\text{for all }x\in\Omega,\ t\geq 1~~\text{and}~~\varepsilon \in (\varepsilon_{j})_{j \in \mathbb{N}}.
\end{equation}
On the other hand, \eqref{2.19*} entails that  for $0<t\leq 1$,
\begin{equation*}
v_\varepsilon(x,t)\geq\|v_0\|_{L^\infty(\Omega)}e^{-\varsigma(1+t)}
\geq  \|v_{\varepsilon}(\cdot,t)\|_{L^{\infty}(\Omega)} e^{-2\varsigma},
\end{equation*}
which along with \eqref{4.19*}  yields
\begin{equation}\label{4.20**}
v_{\varepsilon}(x,t)\geq c_1\|v_{\varepsilon}(\cdot,t)\|_{L^{\infty}(\Omega)}\quad\text{for all }x\in\Omega,\ t> 0~~\text{and}~~\varepsilon \in (\varepsilon_{j})_{j \in \mathbb{N}}
\end{equation}
where $c_1:=\min\{\lambda_*, e^{-2\varsigma}\}$.
Moreover, in view of \eqref{2.12}, we have
\begin{align}\label{4.20*}
L_\varepsilon=\int_{0}^{\infty} \|v_{\varepsilon}(\cdot,t)\|_{L^{\infty}(\Omega)}\leq e^{2\kappa}\|v_0\|_{L^{\infty}(\Omega)} \int_0^{\infty}e^{-\kappa s}ds\leq\frac{e^{2\kappa}}{\kappa}\|v_0\|_{L^{\infty}(\Omega)}.
\end{align}
Therefore, \eqref{4.20**}, \eqref{4.20*} and \eqref{2.19*} provide positive constants $c_2$ and $c_3$ such that
\begin{equation}
a_{\varepsilon}(x,\tau)\leq L_\varepsilon\leq c_2,
\end{equation}
and
\begin{equation}
a_{\varepsilon}(x,\tau)\geq c_1L_\varepsilon=
c_1\int_0^{\infty}\|v_\varepsilon(\cdot,s)\|_{L^{\infty}(\Omega)}\geq c_1c_3\int_0^{\infty} e^{-\varsigma s}\geq\frac{c_1c_3}{2\varsigma}
\end{equation}
for all $(x,\tau)\in \Omega \times (0,1)$ and $\varepsilon \in (\varepsilon_{j})_{j \in \mathbb{N}}$.
Furthermore, due to \eqref{1.8}, there exists $c_4> 0$ such that
\begin{align}\label{4.2*}
   |b_{\varepsilon}(x, \tau)| \leq c_4~~\mbox{and}~~\ell a_\varepsilon(x,\tau)w_\varepsilon\leq c_4~~\text{for all}~~(x,\tau) \in \Omega \times (0,1)~~ \text{and}~~\varepsilon \in (\varepsilon_{j})_{j \in \mathbb{N}}.
\end{align}
By a similar manner in \cite{Li Winkler, Winkler3, Winkler5}, we have
\begin{align}\label{4.3*}
    L_{\varepsilon} \to L:=\int_{0}^{\infty} \|v(\cdot, t)\|_{L^{\infty}(\Omega)} dt, \quad \varepsilon= \varepsilon_{j}\to0.
\end{align}
Hence according to \eqref{2.44} and \eqref{4.3*}, we have
$$\phi_{\varepsilon}(t)\to\phi(t)~~\mbox{for all}~~t>0~~\mbox{as}~~\varepsilon=\varepsilon_{j}\to 0.$$
Moreover, from \eqref{2.43}--\eqref{2.45}, we obtain
\begin{align}
m_{\varepsilon}(x,\tau)\to u(x,\phi^{-1}(\tau)), a_{\varepsilon}(x,\tau)\to a(x,\tau)~~\mbox{and}~~
b_{\varepsilon}(x,\tau)\to b(x,\tau)
\end{align}
for all $(x,\tau)\in\Omega\times(0,1)$ as $\varepsilon=\varepsilon_{j}\to  0$. On the other hand, from \eqref{4.17*} and \eqref{3.39}, we have
\begin{equation*}
\|m_{\varepsilon}(\cdot,\tau)\|_{L^\infty(\Omega)}\leq c_5~~\mbox{for all}~~\tau\in(0,1)~~\mbox{and}~~\varepsilon\in(\varepsilon_j)_{j\in\mathbb{N}}.
\end{equation*}
In view of the bounds in \eqref{4.2*} and the boundedness of $w_\varepsilon$ in \eqref{2.15}, we may rely on the H\"older regularity in quasilinear degenerate parabolic equations (\cite{Holder}) to claim that there exist $\theta\in(0,1)$ and $c_6>0$ such that
\begin{align}
  \|m_{\varepsilon}\|_{C^{\theta,\frac{\theta}{2}}
  (\overline{\Omega}\times[0,1])}\leq c_6\quad\mbox{for all} ~~\varepsilon\in(0,1).
\end{align}
 Then by the Arzel$\grave{a}$-Ascoli theorem, we obtain that
$$m_{\varepsilon}(x,\tau)\to m(x,\tau)\quad\text{in }\ \ C^{0}\left(\overline{\Omega}\times[0,1]\right)\ \text{ as }\varepsilon=\varepsilon_{j}\to 0$$
for some $m\in C^{0}\left(\overline{\Omega}\times[0,1]\right)$. Then we can conclude that
$$m(x,\tau)=u(x,\phi^{-1}(\tau))\quad\text{ for all }\ (x,\tau)\in\Omega\times(0,1),$$
which together with \eqref{4.9} implies that
\begin{align}
u(\cdot,t)\to u_{\infty}~~\mbox{in}~~L^{\infty}(\Omega)~~\mbox{as}~~t\to\infty.
\end{align}
\end{proof}
Thanks to the convergence of $u$ in $L^{\infty}(\Omega)$ achieved in Lemma \ref{lemma4.5*},  
asymptotic behaviour of  the solution exponent $w$ 
can be established through the application of the Neumann heat semigroup smoothing estimates.
\begin{lemma}\label{lemma4.5}
Let $u_{\infty}$ be as defined in Lemma \ref{lemma4.3},
$w_{\infty}$ is the solution to the following elliptic equation
\begin{equation}\label{4.15}
-\Delta w_{\infty}+w_{\infty}=u_{\infty},~~\partial_{\nu}w_{\infty}=0,
\end{equation}
where $\partial_{\nu}$ denotes the normal derivative on the boundary. Then we have
\begin{equation}\label{4.16}
w(\cdot,t)\to w_{\infty}~~\mbox{in}~~W^{1,\infty}(\Omega)~~\mbox{as}~~t\to\infty.
\end{equation}
\end{lemma}
\begin{proof}
According to the $w$-equation and \eqref{4.15}, we have
\begin{equation}\label{4.20}
(w(\cdot,t)-w_{\infty})_t=\Delta(w(\cdot,t)-w_{\infty})
-(w(\cdot,t)-w_{\infty})+(u(\cdot,t)-u_{\infty}).
\end{equation}
Now the variation of constants formula associated with \eqref{4.20} represents $w-w_{\infty}$ according to
\begin{equation}\label{4.21}
\begin{split}
w(\cdot,t)-w_{\infty}= e^{(\Delta-1)(t-t_0)}(w(\cdot,t_0)-w_{\infty})+\int_{t_0}^te^{(t-s)(\Delta-1)}(u(\cdot,s)-u_{\infty}).
\end{split}
\end{equation}
Therefore, applying the Neumann heat semigroup estimates, there exist constants $c_i>0(i=1,2)$ such that for all $t\geq t_0$,
\begin{equation}
\begin{split}
&\|\nabla(w-w_{\infty})(\cdot,t)\|_{L^{\infty}(\Omega)}\\
&\leq \|\nabla e^{(\Delta-1)(t-t_0)}(w-w_{\infty})(\cdot,t_0)\|_{L^{\infty}(\Omega)}
+\int_{t_0}^t\| \nabla e^{(\Delta-1)(t-s)}(u-u_{\infty})(\cdot,s)\|_{L^{\infty}(\Omega)}\\
&\leq c_1 (1+(t-t_0)^{-\frac1 2}) e^{-(\lambda+1)(t-t_0)}\|w(\cdot,t_0)\|_{L^{\infty}(\Omega)}\\
&~~+c_1\int_{t_0}^t (1+(t-s)^{-\frac1 2})e^{-(\lambda+1)(t-s)}\|(u-u_{\infty})(\cdot,s)\|_{L^{\infty}(\Omega)}\\
&\leq c_2  (1+(t-t_0)^{-\frac1 2}) e^{-(t-t_0)}+c_2\sup\limits_{s\geq t_0}\|(u-u_{\infty})(\cdot,s)\|_{L^{\infty}(\Omega)}\cdot\int_{0}^{\infty} (1+\sigma^{-\frac1 2}) e^{-\sigma} d\sigma\\
\end{split}
\end{equation}
where we use $\int_0^{\infty}(1+\sigma^{-\frac{1}{2}})e^{-\sigma}d\sigma<\infty$. Moreover, due to \eqref{4.16*}, we obtain $$\lim\limits_{t\to\infty}\|w(\cdot,t)-w_{\infty}\|_{W^{1,\infty} (\Omega)}=0.
$$
\end{proof}
The decay estimate established in \eqref{2.12} can  hence be applied to derive the following asymptotic decay behavior for the second solution component.
\begin{lemma}\label{lemma4.6}
Let $K>0$ with the property that \eqref{1.6} holds.  Then  we have
\begin{align}\label{4.25}
    v(\cdot,t)\rightarrow 0 ~~\mbox{in}~~W^{1,p}(\Omega)~~\mbox{for all}~~p\geq 1\qquad \mbox{as}~~ t\to \infty.
\end{align}
\end{lemma}
\begin{proof}
Thanks to \eqref{2.12}, we obtain that for all $p>1$
\begin{align}\label{4.26}
    \|v(\cdot,t)\|_{L^{p}(\Omega)}\leq \|v_0\|_{L^\infty(\Omega)}e^{-\kappa (t-1)}|\Omega|^{\frac1 p} 
    \rightarrow 0 \quad \text{as} \quad t \to \infty.
\end{align}
Furthermore, multiplying the second equation in \eqref{2.5} by $v_{\varepsilon}$, we obtain
\begin{equation*}
\frac d{dt}\int_\Omega v^2(\cdot,t)+\int_\Omega |\nabla v|^2=-\int_\Omega v^2w\leq0,
\end{equation*}
which yields
\begin{equation}\label{4.27}
\int_0^{\infty}\int_\Omega|\nabla v|^2\leq\infty.
\end{equation}
Define $f(t) := \int_{\Omega}|\nabla v(x,t)|^2 $. By \eqref{4.27} , it follows that
\begin{align*}
\int_{t-1}^{t} f(s) = \int_{t-1}^{t} \int_\Omega |\nabla v|^2 \rightarrow 0 \quad \text{as} \quad t \to \infty.
\end{align*}
Due to the continuity of $f(t)$, we arrive at
\begin{align}\label{4.28}
\int_{\Omega}|\nabla v(x,t)|^2\rightarrow 0 \quad \text{as} \quad t \to \infty.
\end{align}
Consequently, for any $p \geq 1$, \eqref{4.28} along with \eqref{1.8}  implies
\begin{align}\label{4.29}
\|\nabla v(\cdot,t)\|_{L^p(\Omega)} \to 0 \quad \text{as} \quad t \to \infty.
\end{align}
Therefore \eqref{4.25} follows from \eqref{4.26} and \eqref{4.29}.
\end{proof}

  Beyond the stabilization  result established in Theorem 1.2 for solutions to
   \eqref{1.1}, we further demonstrate that  the  limiting  profile $u_{\infty}$ of the solution component $u$  obtained in \eqref{4.16*} becomes non-homogeneous  when the initial signal concentration $v_0$ is sufficiently small, provided that $u_0$ is not identically constant.

\begin{lemma}\label{lemma4.7}
Let $K > 0$ be such that \eqref{1.6} holds and suppose $u_0\not\equiv$ const.. Then there exists $\delta_2=\delta_2(K,u_0)>0$ such that if $(u_0,v_0,w_0)$ satisfy \eqref{1.2} and \eqref{1.6}, as well as
\begin{align*}
    \int_{\Omega}v_0< \delta_2,
\end{align*}
the corresponding limit function $u_{\infty}\in C(\Omega)$ from Lemma \ref{lemma4.5*} satisfies
$u_{\infty} \not\equiv$ const..
\end{lemma}
\begin{proof}
     Since $u_{0}$ is continuous and not constant, we can find constant $c_{1}>0$, as well as a open set $\Omega_0\subset\Omega$ such that
     \begin{equation}\label{4.39}
     u_{0}(x)+c_1\leq\overline{u_0}:=\frac1{|\Omega|}\int_\Omega u_0~~\mbox{for all}~~x\in\Omega_0,
      \end{equation}
      and thereafter fix some nonnegative $\psi\in C^{\infty}_{0}(\Omega)$ such that $\operatorname{supp}\psi\subset \Omega_0$ and $\int_\Omega \psi=1$.
      At this position, we claim that
     \begin{align}\label{4.41}
        u_{\infty}\equiv c_2~~\mbox{in}~~\Omega~~\mbox{for some}~~c_2\geq0
    \end{align}
    is absurd. To achieve this, assuming \eqref{4.41} to be valid we see that
    \begin{equation}\label{4.42}
    \int_\Omega u_{\infty}\psi=c_2\int_\Omega\psi.
    \end{equation}
      Now, applying Lemma \ref{lemma4.4} to $\eta:=\frac{c_1}{2\|\psi\|_{W^{1,\infty}(\Omega)}}$, there exists $\delta_2>0$ such that whenever $\int_\Omega v_0\leq\delta_2$,
      \begin{equation}\label{4.30}
      \|u_\infty-u_0\|_{(W^{1,\infty}(\Omega))^*}\leq \frac{c_1}{2\|\psi\|_{W^{1,\infty}(\Omega)}},
      \end{equation}
      and hence, in particular,
      \begin{equation}\label{4.44}
      \int_{\Omega}(u_{\infty}-u_0)\psi\leq\|u_{\infty}-u_0\|_{(W^{1,\infty}(\Omega))^*}
      \cdot\|\psi\|_{W^{1,\infty}(\Omega)}\leq \frac{c_1}{2}.
      \end{equation}
      In addition, according to our choice of $\psi$ and \eqref{4.39}, we have
      \begin{equation*}
      \int_\Omega u_0\psi=\int_{\Omega_0} u_0\psi\leq (\overline{u_0}-c_1)\int_\Omega \psi=\overline{u_0}-c_1,
      \end{equation*}
     which along with \eqref{4.42} and \eqref{4.44} yields
      \begin{equation}\label{4.45}
      c_2=\int_\Omega u_{\infty}\psi\leq\int_\Omega u_0\psi+\frac{c_1}{2}
      \leq(\overline{u_0}-c_1)+\frac{c_1}2
      =\overline{u_0}-\frac{c_1}2.
      \end{equation}
      On the other hand, from \eqref{2.8} and \eqref{2.41} it follows that
      \begin{equation*}
      \int_\Omega u(\cdot,t)\geq\int_\Omega u_0~~\mbox{for all}~~t>0,
      \end{equation*}
which along with \eqref{4.16*} implies that
$      \int_\Omega u_{\infty}\geq\int_\Omega u_0,
      $
      and thus
      $
      c_2\geq\overline{u_0}
      $.
This contradicts \eqref{4.45}, and thereby  $u_{\infty}$
 cannot coincide with any constant.
\end{proof}
\begin{proof}[Proof of Theorem \ref{Th1.3}] The non-triviality of the limiting profile $u_\infty$ in Theorem 1.3,
 along with the properties stated in \eqref{1.10},  follows directly from Lemma \ref{lemma4.5}--Lemma \ref{lemma4.7}, which completes the proof of this theorem.
\end{proof}

\vspace{.3cm}

\noindent{\bf Conflict of interest}:
No potential conflict of interest is reported by the authors.

\noindent{\bf Ethics approval}:
Ethics approval is not required for this research.


\noindent{\bf Data availability statement}:
All data that support the findings of this study are included within the article. 

\noindent{\bf Financial support}: This work was supported by the National Natural Science Foundation of China under Grant(12071030, 12271186, 12171498).

\bibliographystyle{elsarticle-num}

\begin{thebibliography}{99}
\bibitem{Blanchet}
A. Blanchet, J.A. Carrillo and Ph. Lauren\c{c}ot, Critical mass for a Patlak-Keller-Segel model with
degenerate diffusion in higher dimensions, Calc. Var. Partial Differ. Equ., 35(2009), 133--68.


\bibitem{Daicvpde(2023)}
F. Dai, How far do indirect signal production mechanisms regularize
the three-dimensional Keller-Segel-Stokes system?, Calc. Var. Part. Differ. Equ., 62(2023), 119.

\bibitem{XiangMao}
X. Deji, A. Huang and Y. Wang, Stabilization of arbitrary structures in a three-dimensional doubly degenerate nutrient taxis system, Preprint.

\bibitem{Fu}
X. Fu, L-H. Tang, C. Liu, J-D. Huang, T. Hwa and P. Lenz,  Stripe formation in bacterial systems
with density-suppressed motility,  Phys. Rev. Lett., 108(2012), 198102.



\bibitem{Fuest(2019)}
M. Fuest, Analysis of a chemotaxis model with indirect signal absorption, J. Differ. Equ., 267(2019), 4778--4806.


\bibitem{Fujie(2017)}
K. Fujie and T. Senba, Application of an Adams type inequality to a two-chemical substances chemotaxis system, J. Differ.  Equ.,
 263(1)(2017), 88--148.



\bibitem{Fujikawa1992}
H. Fujikawa, Periodic growth of bacillus subtilis colonies on agar plates, Phys. A, 189(1992), 15--21.

\bibitem{Fujikawa1989}
H. Fujikawa and M. Matsushita, Fractal growth of bacillus subtilis on agar plates, J. Phys. Soc. Japan, 47(1989), 2764--2767.

\bibitem{Herrero(1997)}
M.A. Herrero and J. L. Vel\'{a}zquez, A blow-up mechanism for a chemotaxis model, Ann. Sc. Norm. Super. Pisa, Cl. Sci., 24(1997), 633--683.

\bibitem{HarnackJDE}
J. H\'{u}ska, Harnack inequality and exponential separation for oblique derivative problems on Lipschitz domains. J. Differ. Equ., 226(2006), 541--557.

\bibitem{Hillen(2009)}
T. Hillen and K. J. Painter, A user's guide to PDE models for chemotaxis, J. Math. Biol., 58(2009), 183--217.

\bibitem{Hillen(2013)}
T. Hillen, K.J. Painter and M. Winkler, Convergence of a cancer invasion model to a logistic chemotaxis model, Math. Models Methods Appl. Sci., 23(2013), 165--198.


\bibitem{Kawasaki1997}
K. Kawasaki, A. Mochizuki, M. Matsushita, T. Umeda and N. Shigesada, Modeling spatio-temporal patterns generated by bacillus subtilis, J. Theor. Biol. 188(1997), 177--185.

\bibitem{Schauder}
O.A. Ladyzenskaja, V.A. Solonnikov and N.N. Ural\'ceva, Linear and Quasi-Linear Equations of Parabolic Type, Amer. Math. Soc. Transl., 23(1968), Providence, RI.

\bibitem{Lankeit}
J. Lankeit, Locally bounded global solutions to a chemotaxis consumption model with singular sensitivity
and nonlinear diffusion, J. Differ. Equ., 262(2017), 4052--4084.

\bibitem{Laurenot}
P. Lauren\c{c}ot, Large time convergence for a chemotaxis model with degenerate local sensing and consumption,  Bull. Korean Math. Soc.,  61(2024), 479--488.


\bibitem{Leyva2013}
J.F. Leyva, C. M\'{a}laga and R.G. Plaza, The effects of nutrient chemotaxis on bacterial aggregation patterns with nonlinear degenerate cross diffusion, Phys. A, 392(2013), 5644--5662.


\bibitem{LiJDE}
G. Li, Large-data global existence in a higher-dimensional doubly degenerate nutrient system, J. Differ. Equ., 329(2022), 318--347.


\bibitem{Li Winkler}
G. Li and M. Winkler, Nonnegative solutions to a doubly degenerate nutrient taxis system, Commun. Pure Appl. Anal., 21(2022), 687--704.

\bibitem{Li Winkler(2022R)}
G. Li and M. Winkler, Refined regularity analysis for a Keller-Segel-consumption system involving signal-dependent motilities, Appl. Anal., 103(2024), 45--64.



\bibitem{Liu}
C. Liu et al., Sequential establishment of stripe patterns in an expanding cell population, Science, 334(2011), 238.


\bibitem{Liu Li Huang(2020)}
Y. Liu, Z. Li and J. Huang, Global boundedness and large time behavior of a chemotaxis system with indirect signal absorption, J. Differ. Equ., 269(2020), 6365--6399.


\bibitem{Fujikawa1990}
M. Matsushita and H. Fujikawa, Diffusion-limited growth in bacterial colony formation, Physica A: Statistical Mechanics and its Appl., 168(1990), 498--506.

\bibitem{Ohgiwari}
M. Ohgiwari, M. Matsushita and  T. Matsuyama, Morphological changes in growth phenomena of bacterial
colony patterns, J. Phys. Soc. Jpn., 50(1992), 705--711.

\bibitem{XP}
X. Pan, Superlinear degradation in a doubly degenerate nutrient taxis system, Nonlinear Anal. Real World Appl., 77(2024), 104040.

\bibitem{Plaza2019}
R. G. Plaza, Derivation of a bacterial nutrient-taxis system with doubly degenerate crossdiffusion as the parabolic limit of a velocity-jump process, J. Math. Biol, 78(2019), 1681--1711.

\bibitem{PS2002}
P. Pol\'{a}$\breve{c}$ik  and F. Simondon, Nonconvergent bounded solutions of semilinear heat equations on arbitrary domains, J. Differ. Equ., 186(2002), 586--610.

\bibitem{PY2003}
P.~Pol\'{a}$\breve{c}$ik and E.~Yanagida, On bounded and unbounded global solutions of a supercritical semilinear heat equation, Math. Ann., 327(2003), 745--771.


\bibitem{Holder} M.M. Porzio and V. Vespri, H\"{o}lder estimates for local solutions of some doubly nonlinear degenerate parabolic equations,  J. Differ. Equ., 103(1)(1993), 146--178.

\bibitem{Strohm(2013)}
S. Strohm, R. C. Tyson and J. A. Powell, Pattern formation in a model for mountain
pine beetle dispersal: linking model predictions to data, Bull. Math. Biol.,  75(2013), 1778--1797.

\bibitem{Surulescu(2021)}
C.~Surulescu  and M.~Winkler,
Does indirectness of signal production reduce the explosion-supporting potential in chemotaxis-haptotaxis systems? Global classical
solvability in a class of models for cancer invasion (and more), Euro. J. Appl. Math., 32(4)(2021), 618--651.


\bibitem{TaoW} Y. Tao and  M. Winkler, Boundedness in a quasilinear parabolic-parabolic Keller-Segel system with subcritical sensitivity, J. Differ. Equ., 252(2012), 692--715.

\bibitem{TaoWinkler(2017)}
Y. Tao and M. Winkler, Critical mass for infinite-time aggregation in a chemotaxis model with indirect signal production, J. Eur. Math. Soc., 19(2017), 3641--3678.



\bibitem{TelloWrzosek(2017)}
J.I. Tello and  D. Wrzosek,  Predator--prey model with diffusion and indirect prey-taxis, Math. Models Methods Appl. Sci., 26(2016), 2129--2162.




\bibitem{Winkler(2024)}
M. Winkler, A quantitative strong parabolic maximum principle and application to a taxis-type migration consumption model involving signal-dependent degenerate diffusion,
Ann. Inst. H. Poincar\'e Anal. NonLin\'eaire, 41(2024), 95--127.

\bibitem{Winkler(2010)} M. Winkler, Aggregation vs. global diffusive behavior in the higher-dimensional Keller-Segel model, J. Differ. Equ., 248(2010), 2889--2905.





\bibitem{WDCDSB} M. Winkler, Approaching logarithmic singularities in quasilinear chemotaxis-consumption systems with signal-dependent sensitivities, Discrete Contin. Dyn. Syst. Ser. B., 27(2022), 6565--6587.

\bibitem{Winkler3} M. Winkler, Does spatial homogeneity ultimately prevail in nutrient taxis systems? A paradigm for structure support by rapid   diffusion decay in an autonomous parabolic flow, Trans. Am. Math. Soc., 374(2021), 219--268.

\bibitem{WinklerMAS2025}
    M. Winkler, Effect of degeneracies in taxis-driven evolution,  Math. Models Methods Appl. Sci., 35(2025), 283--343.

\bibitem{Winkler5} M. Winkler, Elliptic Harnack inequalities in linear parabolic equations and application to the asymptotics in a doubly degenerate nutrient taxis system, Preprint.

\bibitem{WinklerEven}
M. Winkler, Eventual regularity in a two-dimensional doubly degenerate nutrient taxis system, Preprint.

\bibitem{Winkler(2013)}
M. Winkler, Finite-time blow-up in the higher-dimensional parabolic-parabolic Keller-Segel system, J. Math. Pures Appl., 100(2013), 748--767.

\bibitem{Winkler(2023Non)}
M. Winkler, Stabilization despite pervasive strong cross-degeneracies in a nonlinear
diffusion model for migration-consumption interaction, Nonlinearity, 36(2023), 4438--4469.

\bibitem{WCVPDE} M. Winkler, Stabilization of arbitrary structures in a doubly degenerate reaction-diffusion system modeling bacterial motion
    on a nutrient-poor agar, Calc. Var. Partial Differ. Equ., 61(2022), 108.

\bibitem{WNARWA} M. Winkler, Small-signal solutions of a two-dimensional doubly degenerate taxis system modeling bacterial motion in
    nutrient-poor environments, Nonlinear Anal. Real World Appl., 63(2022), 1468--1218.

\bibitem{WJDDE}
M.~Winkler, Large time behavior and stability of equilibria of degenerate parabolic equations, J. Dyn. Differ. Equ. 17(2005), 331--351.


\bibitem{WJDE} M. Winkler, $L^{\infty}$ bounds in a two-dimensional doubly degenerate nutrient taxis system with general cross-diffusive flux, J. Differ. Equ., 400(2024), 423--456.



\bibitem{Wu} D. Wu, Refined existence theorems for doubly degenerate chemotaxis-consumption systems with large initial data, Nonlinear Differ. Equ. Appl., 31(2024), 104.



\bibitem{Xing(2021)}
J. Xing, P. Zheng, Y. Xiang and H. Wang, On a fully parabolic singular chemotaxis--(growth) system with indirect signal production or consumption, Z. Angew. Math. Phys.,  72(2021), 105.



\bibitem{LiWinkler2}
G. Li and M. Winkler, Continuous solutions for a two-dimensional cross-diffusion problem involving doubly degenerate diffusion and logistic proliferation,  Anal. Appl., 23(2025), 485--510.

\bibitem{ZhangLi}
Z. Zhang and Y. Li, Boundedness in a two-dimensional doubly degenerate nutrient taxis system, Preprint.





\end{thebibliography}

\end{document}